\documentclass[12pt,reqno]{amsart}

\usepackage{amsmath} 
\usepackage{amsfonts}
\usepackage{amssymb}
\usepackage{mathrsfs}
\usepackage{asymptote}
\usepackage[top=1.25in, bottom=1in, left=1in, right=1in]{geometry}
\usepackage{etoolbox}
\usepackage[T1]{fontenc}
\usepackage{enumitem}   
\usepackage{hyperref}

\hypersetup{
    colorlinks=false,
}

\newtheorem{theorem}{Theorem}[section]
\newtheorem{definition}[theorem]{Definition}
\newtheorem{lemma}[theorem]{Lemma}

\makeatletter

\renewcommand\subsubsection{\@secnumfont}{\bfseries}%
\renewcommand\subsubsection{\@startsection{subsubsection}{3}
  \z@{.5\linespacing\@plus.7\linespacing}{-.5em}%
  {\normalfont\bfseries\itshape}}
  
\makeatother

\newcommand{\ds}{\displaystyle}

\renewcommand{\epsilon}{\varepsilon}

\renewcommand{\d}[3][]{
    \ifstrempty{#1}{
        \frac{d#2}{d#3}
    }{
        \frac{ d^{#1}#2 }{ {d#3}^{#1} }
    }
}
\newcommand{\p}[2]{\frac{\partial#1}{\partial#2}}

\newcommand{\sgn}{\operatorname{sgn}}

\newcommand{\Var}{\operatorname*{Var}}
\newcommand{\dint}{\displaystyle\int}
\newcommand{\dlim}{\displaystyle\lim}

\newcommand{\leb}{\mathcal{L}}
\newcommand{\essVar}{\operatorname*{essVar}}
\newcommand{\Glim}{\Gamma\text{-}\lim}

\makeatletter
    \newcommand{\wc}{\@ifstar{\wcstar}{\wcnostar}}
\makeatother

\newcommand{\wcnostar}{\rightharpoonup}
\newcommand{\wcstar}{\overset{\ast}{\rightharpoonup}}

\newcommand{\R}{\mathbb{R}}

\newcommand{\N}{\mathbb{N}}

\usepackage{indentfirst}
\title[Boundary Conditions for Second-Order Phase Transitions]{Effect of Boundary Conditions on Second-Order Singularly-Perturbed Phase Transition Models on $\mathbb{R}$}
\author{Thomas Lam}
\date{}

\begin{document}
\maketitle

\begin{abstract}
The second-order singularly-perturbed problem concerns the integral functional $\int_\Omega \epsilon_n^{-1}W(u) + \epsilon_n^3\|\nabla^2u\|^2\,dx$ for a bounded open set $\Omega \subseteq \R^N$, a sequence $\epsilon_n \to 0^+$ of positive reals, and a function $W:\R \to [0,\infty)$ with exactly two distinct zeroes.  This functional is of interest since it models the behavior of phase transitions, and its Gamma limit as $n \to \infty$ was studied by Fonseca and Mantegazza.  In this paper, we study an instance of the problem for $N=1$.  We find a different form for the Gamma limit, and study the Gamma limit under the addition of boundary data.
\end{abstract}

\section{Introduction}

For a bounded open set $\Omega \subseteq \R^N$, we may imagine $\Omega$ as a container for a liquid whose density is given by  $u:\Omega \to \R$.  The potential energy of the liquid can be measured by the integral functional $u \mapsto \int_\Omega W(u)\,dx$ where $W:\R \to [0,\infty)$ is the energy per unit volume.

Suppose that $W$ is a \textit{two-welled potential}, so that $W$ has exactly two distinct zeroes $z_1,z_2$ with $z_1 < z_2$, which are the \textit{phases} of the liquid.  Then the liquid will tend to take on the two densities $z_1$ and $z_2$, and in particular its density will simply take the form $u = z_1 \cdot 1_E + z_2 \cdot 1_{\Omega \setminus E}$ in order to minimize the potential $\int_\Omega W(u)\,dx$, which is unexciting.

However, in practice, such rapid changes in density induces high \textit{interfacial energy} between phases.  To account for this, the Van der Waals-Cahn-Hilliard theory of phase transitions \cite{cahnhilliard} models the potential energy via an integral functional of the form
\begin{equation}\label{eq:Jfunctional}
    J_\epsilon(u) := \int_\Omega W(u) + \epsilon^2\|\nabla u\|^2\,dx, \qquad u \in W^{1,2}(\Omega).
\end{equation}
The problem of interest is to minimize $J_\epsilon(u)$ subject to a mass constraint $\int_\Omega u\,dx = m$.  A minimizing function $u_\epsilon$ gives a stable density distribution that the liquid would likely conform to.

What sort of stable density distribution is approached as $\epsilon \to 0^+$?  Specifically, suppose that we have a sequence $\epsilon_n \to 0^+$ such that the sequence of minimizers $u_{\epsilon_n}$ converges to a function $u$ in some reasonable sense.  What properties must be satisfied by $u$?

Using Gamma convergence (described in Section 2.4), Modica \cite{modica} and Sternberg \cite{sternberg} independently proved that such a $u$ satisfies $u \in \{z_1,z_2\}$ almost everywhere, and minimizes the perimeter of $u^{-1}(z_1)$.  In particular, they prove that for each sequence of positive reals $\epsilon_n$ with $\epsilon_n \to 0^+$, we have that
$$\Glim_{n \to \infty} (J_{\epsilon_n}/\epsilon_n)(u) = \begin{cases}2\dint_{z_1}^{z_2} \sqrt{W(x)}\,dx \cdot \operatorname{Per}_\Omega(u^{-1}(z_1)), & u \in BV(\Omega;\{z_1,z_2\}) \\ +\infty, & \text{otherwise}\end{cases}$$
under $L^1(\Omega)$ convergence, where $\operatorname{Per}_\Omega(E)$ denotes the \textit{perimeter} of a set $E \subseteq \Omega$, defined generally as
$$\operatorname{Per}_\Omega(E) := \sup\left\{\int_E \operatorname{div} \varphi : \varphi \in C_0^\infty(\R^N), |\varphi| \leq 1\right\}.$$

Owen, Rubinstein and Sternberg \cite{boundary} studied the family of functionals $J_\epsilon$ under boundary conditions instead of a mass constraint.  They prove that if $\Omega$ has $C^2$ boundary, $h_\epsilon \in L^p(\partial \Omega) \cap L^\infty(\partial \Omega)$ is the trace of a function in $W^{1,2}(\Omega)$ for each $\epsilon > 0$, $h_\epsilon \to h \in L^1(\partial \Omega) \cap L^\infty(\partial \Omega)$ in $L^1(\partial \Omega)$ as $\epsilon \to 0$, $\int_{\partial \Omega} \left|\p{h_\epsilon}{\sigma}\right|$ is bounded in $\epsilon$ and $\left\|\p{h_\epsilon}{\sigma}\right\|_{L^\infty(\partial \Omega)} \leq C\epsilon^{-1/4}$ for a constant $C> 0$ where $\sigma$ is a surface parameter on $\partial \Omega$, and $K_\epsilon:L^1(\Omega) \to \overline{\R}$ is defined as
\begin{equation}\label{eq:Kfunctional}
  K_\epsilon := \begin{cases}\int_\Omega \epsilon^{-1}W(u) + \epsilon\|\nabla u\|^2\,dx, & u \in W^{1,2}(\Omega) \text{ and } \operatorname{Tr} u = h_\epsilon \\ +\infty, & \text{otherwise}\end{cases},  
\end{equation}
where $\overline{\mathbb{R}} = [-\infty,\infty]$, then
\begin{align*}
    \Glim_{n \to \infty} &K_{\epsilon_n}(u) \\
    &= \begin{cases}\int_\Omega |\nabla \chi(u)| + \int_{\partial \Omega} |\chi(h(x)) - \chi(\operatorname{Tr}(u)(x))|\,d\mathcal{H}^{N-1}(x), & u \in BV(\Omega;\{z_1,z_2\}) \\ +\infty, & \text{otherwise} \end{cases},
\end{align*}
where
$$\chi(t) := 2\int_{z_1}^t \sqrt{W(z)}\,dz$$
and $\epsilon_n \to 0^+$.


Fonseca and Mantegazza \cite{fonseca} consider a second-order derivative.  To be precise, they define
\begin{equation}\label{eq:Hfunctional}
    H_\epsilon(u) := \int_\Omega \epsilon^{-1}W(u) + \epsilon^3\|\nabla^2 u\|^2\,dx
\end{equation}
and proved that for every sequence of positive reals $\epsilon_n \to 0^+$, we have 
$$\Glim_{n \to \infty} H_{\epsilon_n}(u) = \begin{cases}c \cdot \operatorname*{Per}_\Omega(u^{-1}(z_1)), & u \in BV(\Omega, \{z_1,z_2\}) \\ +\infty, \text{otherwise}\end{cases}$$
under $L^1(\Omega)$ convergence, where
\begin{equation}c := \min\left\{\int_\R W(u) + |u''|^2\,dt : u \in W^{2,2}_\text{loc}(\R), \lim_{t \to -\infty} u(t) = z_1, \lim_{t \to \infty} u(t) = z_2\right\}. \label{eq:c}\end{equation}

For a treatment of more general functionals, see \cite{chermisi},  \cite{cicalese}, and \cite{boundaryy}.

Fonseca and Mantegazza's proofs appeal to rather sophisticated constructs such as the Young measure.  In this paper, we consider a 1-dimensional instance of the problem solved by Fonseca and Mantegazza, which will allow for alternative and more elementary methodologies.   

Our goal will ultimately be to combine the efforts of Owen, Rubinstein, and Sternberg with those of Fonseca and Mantegazza by considering the addition of boundary conditions as in \ref{eq:Kfunctional} to the second-order problem as in \ref{eq:Hfunctional}, which will be possible due to our alternative methodologies.

To wit, define
\begin{equation}
  \Phi(u) := \left(\int_0^1 (u(y)^2-1)^2\,dy\right)^{3/4}\left(\int_0^1|u''(y)|^2\,dy\right)^{1/4} \label{eq:Phi}
\end{equation}
for all $u \in W^{2,p}(0,1)$.  Define the families
\begin{equation}
    \mathscr{J} := \{u \in W^{2,p}(0,1) : u(0_+) = -1, u(1_-)  = 1, u'(0_+) = u'(1_-) = 0\} \label{eq:famJ}
\end{equation}
and
\begin{equation}
    \mathscr{J}'(t) := \{u \in W^{2,p}(0,1) : u(0_+) = -1, u(1_-)  = t, u'(0_+) = 0\} \label{eq:famJ2}
\end{equation}
for each $t \in \R$.  Let 
\begin{equation}
 \alpha := \frac{2}{3^{3/4}}\inf_{u \in \mathscr{J}} \Phi(u) \label{eq:alpha}
\end{equation}
and 
\begin{equation}
\beta(t) := \frac{4}{3^{3/4}}\inf_{u \in \mathscr{J}'(t)} \Phi(u) \label{eq:beta}    
\end{equation}
for all $t \in \R$.  Then our main result is the following theorem.

\begin{theorem}\label{boundary2}
Let $\Omega = (a,b)$ and let $1 \leq p \leq 4$.  Let $a_0,b_0 \in \R$, and for each $\epsilon > 0$, let $a_\epsilon, b_\epsilon \in \R$ be such that $a_\epsilon \to a_0$ and $b_\epsilon \to b_0$ as $\epsilon \to 0^+$.  For each $\epsilon > 0$ define a functional $G_\epsilon(u):L^p(\Omega) \to \overline{\mathbb{R}}$ via
$$G_{\epsilon}(u) := \begin{cases}\int_\Omega \epsilon^{-1}(u^2-1)^2 + \epsilon^3|u''|^2\,dx, & u \in W^{2,p}(\Omega),u(a_+) = a_\epsilon, u(b_-) = b_\epsilon \\ +\infty, & \text{otherwise}\end{cases}.$$
Let $\epsilon_n \to 0^+$ be a sequence of positive reals.  Then, under strong $L^p(\Omega)$ convergence, we have that
\begin{align*}
  \Glim_{n \to \infty}&G_{\epsilon_n}(u) \\&= \begin{cases}\alpha\essVar_\Omega u + \beta(-a_0\sgn u(a_+)) + \beta(-b_0 \sgn u(b_-)), & u \in \text{BPV}(\Omega;\{-1,1\}) \\ +\infty, & \text{otherwise}\end{cases}.  
\end{align*}
\end{theorem}
Here, $\essVar$ denotes \textit{essential variation}, which is described in Section 2.2.  The addition of boundary conditions is of particular interest here because they can ensure at least one phase transition by preventing the existence of trivial minimizers.

The utility of this Gamma convergence result is justified by the following compactness result. 

\begin{theorem}[Compactness for Second Order Problem]\label{compactness2}
Let $\Omega = (a,b)$ and let $1 \leq p \leq 4$.  For each $\epsilon > 0$, define the functional $F_\epsilon:L^p(\Omega) \to \overline{\R}$ via
\begin{equation}
    F_{\epsilon}(u) := \begin{cases}\int_\Omega \epsilon^{-1}(u^2-1)^2 + \epsilon^3|u''|^2\,dx, & u \in W^{2,p}(\Omega) \\ +\infty, & \text{otherwise}\end{cases}. \label{eq:Fdef}
\end{equation}
Let $\epsilon_n$ be a sequence of positive reals with $\epsilon_n \to 0^+$.  If we have a sequence $u_n \in L^p(\Omega)$ with $\sup_{n \in \N}|F_{\epsilon_n}(u_n)| < \infty$, then there exists a subsequence $u_{n_k}$ for which $u_{n_k} \to u$ in $L^p(\Omega)$ for some $u \in BPV(\Omega;\{-1,1\})$.
\end{theorem}

Proofs of this result are given in \cite{fonseca}, and \cite{cicalese}.  We will provide yet another proof.

The structure of this paper is as follows.  In Section 2, we review pointwise variation and results in Sobolev spaces.  We then define Gamma convergence and motivate its study.  

In Section 3, we prove Theorem \ref{compactness2}, and then give a more elementary proof for Fonseca and Mantegazza's results in the one-dimensional case.  Specifically, let $\Omega = (a,b)$ and define the integral functional $F_\epsilon : L^p(\Omega) \to \overline{\R}$ as in \eqref{eq:Fdef}.  Let $\epsilon_n$ be a sequence of positive reals with $\epsilon_n \to 0^+$, and let $1 \leq p \leq 4$.  We first prove that if $u_n \in L^2(\Omega)$ is a sequence for which $\sup_{n \in \N} F_{\epsilon_n}(u_n) < \infty$, then we can find a subsequence $u_{n_k}$ such that $u_{n_k} \to u$ in $L^p(\Omega)$ for some $u \in BPV(\Omega;\pm 1)$.  Then, we prove that for $1 \leq p < \infty$, we have
$$\Glim_{n \to +\infty} F_{\epsilon_n}(u) = \begin{cases}\alpha\essVar_\Omega u, & u \in BPV(\Omega;\pm 1) \\ +\infty, & \text{otherwise}\end{cases}$$
under $L^p(\Omega)$ convergence, where $\alpha$ is defined as in $\eqref{eq:alpha}$.

Lastly, in Section 4, we build off of the work done in Section 3 to prove our main result.

The current work, and a \textit{slicing} methodology, will be used to extend to the $N$-dimensional case. \cite{me}

\section{Preliminaries}








\subsection{Pointwise Variation and Essential Variation}

For a function $u:I \to \R$, where $I$ is an interval, we may define the pointwise variation of $u$ as
$$\Var_I u := \sup\left\{\sum_{i=1}^n |u(x_i)-u(x_{i-1})|, x_i \in I, x_0 < x_1 < \ldots < x_n\right\}.$$
If $\Var_I u < \infty$ then we write $u \in BPV(I)$.  A useful property is that if $u$ is absolutely continuous, then $\Var_I u = \int_I |u'|\,dx$

A family of functions having uniformly bounded pointwise variation is a powerful property.

\begin{theorem}[Helly's Selection Theorem]
Let $I$ be an interval and $\mathcal{F} \subseteq BPV(I)$ be an infinite family of functions $u:I \to \R$ such that $\sup_{u \in \mathcal{F}} \Var_I u \leq C$ for a constant $C > 0$.  Assume moreover that there exists $x_0 \in I$ such that the set $\{u(x_0) : u \in \mathcal{F}\}$ is bounded.  Then there exists a sequence $\{u_n\}_n \in \mathcal{F}$ that converges pointwise to some $u \in BPV(I)$.
\end{theorem}

This is given as Theorem 2.44 in \cite{leonibook}, where a proof can be found.

In the case that such a pointwise convergence is obtained, we can moreover obtain a bound on $\Var_I u$.

\begin{theorem}\label{hellybound}
If $u_n,u \in BPV(I)$ and $u_n \to u$ pointwise, then
$$\liminf_{n \to \infty} \Var_I u_n \geq \Var_I u.$$
\end{theorem}

See Proposition 2.47 in \cite{leonibook} for a proof.

Now suppose, say, $u \in L^p(I)$.  Then there is no sense in speaking of $\Var_I u$ because the values of $u$ are not well-defined pointwise.  The workaround is to define the \textit{essential pointwise variation} of $u$.

\begin{definition}[Essential Pointwise Variation]
Let $I$ be an interval and $u \in L^1_{\text{loc}}(I)$.  Then the essential pointwise variation of $u$ over $I$ is given by
$$\essVar_I u := \inf\left\{\Var_I \tilde{u} : \tilde{u} \text{ is a representative of }u\right\}.$$
\end{definition}

Particularly, if $u \in W^{1,1}(I)$ then $u$ has an absolutely continuous representative $\tilde{u}$ and we can show that $\essVar_I u = \int_I |\tilde{u}'|\,dx$.  A nice property is that the infimum in the definition for $\essVar_I$ is obtained, i.e. there is always a representative $\tilde{u}$ for $u$ such that $\essVar_I u = \Var_I \tilde{u}$.  This can be proven using Helly's Selection Theorem. 

For $u \in L^p(I)$, we write $u \in BPV(I)$ if $u$ has a representative $\tilde{u}$ with $\tilde{u} \in BPV(I)$.  By the property from above, we see that $u \in BPV(I)$ if and only if $\essVar_I u < \infty$. 

\subsection{Sobolev Spaces}

We begin by defining weak differentiation.

\begin{definition}[Weak Derivative]
Let $\Omega \subseteq \R$ be an open set and $1 \leq p \leq \infty$.  For $u \in L^p(\Omega)$, we say that $u$ admits a weak derivative of order $k \in \N$ if there exists $v \in L^p(\Omega)$ satisfying
$$\int_\Omega u\varphi^{(k)}\,dx = (-1)^k\int_\Omega v\varphi\,dx$$
for all $\varphi \in C^\infty_c(\Omega)$.
\end{definition}

It is not too difficult to verify that the weak derivative is unique up to almost-everywhere equivalence.  If $u$ has a differentiable representative, we call its derivative (in the traditional sense) the \textit{strong} derivative, so as to distinguish the two notions of derivative.

The weak derivative is also notated in the same way as the strong derivative.  For example, if $\Omega = (-1,1)$ and we take $u(x) = |x|$, then $u$ admits a weak derivative given by $u'(x) = \sgn(x)$.

Sobolev spaces consist of $L^p$ functions that admit weak derivatives.

\begin{definition}[Sobolev Space]
Let $\Omega \subseteq \R$ be open, $1 \leq p \leq \infty$, and $k \in \N$.  Then the Sobolev space $W^{k,p}(\Omega)$ is the normed space of all $u \in L^p(\Omega)$ that admit weak derivatives up to order $k$, such that $u^{(l)} \in L^p(\Omega)$ for all $1 \leq l \leq k$.  We may endow $W^{k,p}(\Omega)$ with the following norm:
$$\|u\|_{W^{k,p}(\Omega)} := \|u\|_{L^p(\Omega)} + \sum_{l=1}^k \|u^{(l)}\|_{L^p(\Omega)}.$$
\end{definition}

This definition may be unwieldy for showing that a function belongs to a Sobolev space, so we often work with the following equivalent condition.

\begin{theorem}[ACL Condition for $k=1$ on $\R$]
Let $\Omega \subseteq \R$ be an open set and $1 \leq p < \infty$.  Then $u \in W^{1,p}(\Omega)$ if and only if $u$ is absolutely continuous and $u, u' \in L^p(\Omega)$.  Moreover, if $u \in W^{1,p}(\Omega)$ then the strong and weak derivatives of $u$ agree.
\end{theorem}

Absolute continuity of $L^p$ functions is discussed in the sense that there exists an absolutely continuous representative, and similarly for differentiability.  A proof of this condition may be found in \cite{leonibook}.

An analogue of this condition exists for $k=2$.

\begin{theorem}[ACL Condition for $k=2$ on $\R$]\label{ACLcondition}
Let $\Omega \subseteq \R$ be an open set and $1 \leq p < \infty$.  Then $u \in W^{2,p}(\Omega)$ if and only if $u \in C^1(\Omega)$, $u'$ is absolutely continuous, and $u,u',u'' \in L^p(\Omega)$.  Moreover, if $u \in W^{2,p}(\Omega)$ then the first and second-order strong derivatives of $u$ agree with their weak analogues.
\end{theorem}

A proof of this may be found in \cite{ACL}.


Various algebraic properties satisfied by strong derivatives have analogues for weak derivatives.  For instance, we have the following chain rule.

\begin{theorem}[Chain Rule]\label{chain}
Let $\Omega \subseteq \R$ be an open, bounded interval.  Suppose $f:\R \to \R$ is Lipschitz and $u \in W^{1,p}(\Omega)$.  Then $f(u) \in W^{1,p}(\Omega)$ and $\d{}{x} f(u) = f'(\tilde{u})u'$, where $\tilde{u}$ is the absolutely continuous representative of $u$, and we take $f'(\tilde{u}(x))u'(x)$ to be 0 whenever $u'(x)= 0$.
\end{theorem}

This is a consequence of Exercise 11.51 in \cite{leonibook}.

It is useful to obtain a bound for $\int_a^b |u'|^2\,dx$ in terms of $\int_a^b |u|^2\,dx$ and $\int_a^b |u''|^2\,dx$ for $u \in W^{2,p}(a,b)$.  The following two results are special cases of Lemma 7.38 and Theorem 7.37 in \cite{leonibook}.

\begin{lemma}\label{inter1}
Let $I = (a,b)$.  Suppose $u \in W^{2,1}(I)$ such that $u'$ has at least one zero in $[a,b]$.  Then there exists a universal constant $c > 0$ such that
$$\int_a^b |u'|^2\,dx \leq c\left(\int_a^b u^2\,dx\right)^{1/2}\left(\int_a^b |u''|^2\,dx\right)^{1/2}.$$
\end{lemma}

\begin{theorem}\label{inter2}
For an open interval $I$ and $u \in W^{2,1}(I)$, there exists a universal constant $c > 0$ such that
$$\left(\int_I |u'|^2 \,dx\right)^{1/2} \leq cl^{-1}\left(\int_I u^2 \,dx\right)^{1/2} + cl\left(\int_I |u''|^2 \,dx\right)^{1/2}$$
for every $l$ with $0 < l < \leb^1(I)$.
\end{theorem}

We note that by applying the QM-AM inequality to the above theorem, we may obtain the inequality
$$\int_I |u'|^2\,dx \leq c'l^{-2}\int_I u^2\,dx + c'l^2\int_I |u''|^2\,dx$$
under the same conditions but for a different universal constant $c' > 0$.

\subsection{Weak Convergence}

\begin{definition}[Weak Convergence in $L^p$]
Let $1 \leq p < \infty$.  Let $E \subseteq \R$ be measurable, and let $u_n \in L^p(E)$ for all $n \in \N$.  For $u \in L^p(E)$, we say that $u_n$ converges weakly to $u$ in $L^p(E)$ if 
$$\lim_{n \to \infty} \int_E u_nv\,dx = \int_E uv\,dx$$
for all $v \in L^{p'}(E)$, where $p' \in [1,\infty]$ is such that $\frac1p+\frac1{p'} = 1$, and we write $u_n \wc u$ in $L^p(E)$.
\end{definition}

\begin{definition}[Weak Convergence in $W^{k,p}$]
For $\Omega \subseteq \R$ open, $k \in \mathbb{N}$, $1 \leq p < \infty$, and $u_n,u \in W^{k,p}(\Omega)$, we say that $u_n$ converges weakly in $W^{k,p}(\Omega)$ if $u_n^{(l)} \wc u^{(l)}$ in $L^p(\Omega)$ for every $0 \leq l \leq k$, and we write $u_n \wc u$ in $W^{k,p}(\Omega)$.
\end{definition}

It can be shown that weak limits are unique up to almost-everywhere equivalence.  Moreover, H\"older's inequality implies that strong convergence in $L^p(E)$ implies weak convergence in $L^p(E)$, and consequently strong convergence in $W^{k,p}(\Omega)$ implies weak convergence in $W^{k,p}(\Omega)$.

Proofs of the theorems that follow may be found in \cite{bigbook}.  An important property is that we may strengthen some convergences in the definition of weak convergence in $W^{k,p}(\Omega)$.

\begin{theorem}\label{weakstronger}
Let $(a,b)$ be an interval, $k \in \N$, $1 \leq p < \infty$, and suppose $u_n,u \in W^{k,p}(a,b)$ are such that $u_n \wc u$ in $W^{k,p}(a,b)$.  Then $u_n^{(l)} \to u^{(l)}$ strongly in $L^p(a,b)$ for all $0 \leq l \leq k-1$.
\end{theorem}



Weak convergence enables us to consider a notion of \textit{weak compactness}.

\begin{theorem}[Weak Compactness in $W^{k,p}$]\label{weakcompactness}
Let $\Omega \subseteq \R$ be open, $k \in \N$, $1 \leq p < \infty$, and suppose that $u_n \in W^{k,p}(\Omega)$ is such that $\{u_n\}_n$ is uniformly bounded on $W^{k,p}(\Omega)$.  That is, $\sup_{n \in \N} \|u_n\|_{W^{k,p}(\Omega)} < \infty$.  Then there exist a subsequence $\{u_{n_ll}\}_l$ and $u \in W^{k,p}(\Omega)$ such that $u_{n_l} \wc u$ in $W^{k,p}(\Omega)$.
\end{theorem}

Lastly, weak convergence is useful for obtaining a weak form of semi-continuity for certain functions.  A particular result we shall use is the following.

\begin{theorem}\label{convexslsc}
Let $1 \leq p < \infty$ and let $\Omega \subseteq \mathbb{R}^N$ be open and bounded. Suppose that $f:\mathbb{R} \to \mathbb{R}$ is convex. Then the integral functional $F:L^p(\Omega) \to \overline{\mathbb{R}}$ defined by
$$F(u) := \int_\Omega f(u)\,dx$$
is sequentially lower semi-continuous with respect to weak $L^p_{\text{loc}}(\Omega)$ convergence.  That is, we have
$$F(u) \leq \liminf_{n \to \infty} F(u_n)$$
whenever $u_n,u \in L^p(\Omega)$  such that $u_n \wc u$ in $L^p_\text{loc}(\Omega)$.
\end{theorem}

For a proof, see Theorem 5.14 in \cite{bigbook}.

\subsection{Gamma Limits}

The Gamma limit is a notion of function convergence that is of interest because it preserves minima.

\begin{definition}
For a metric space $(X,d)$ and a sequence $f_n:X \to \overline{\R}$, we say that $f_n$ \textbf{Gamma converges} to a function $f:X \to \overline{\R}$ if for all $x_0 \in X$:
\begin{enumerate}
    \item $\liminf_{n \to \infty} f_n(x_n) \geq f(x_0)$ for \textbf{all} sequences $x_n$ with $x_n \stackrel d\to x_0$, and
    \item $\limsup_{n \to \infty} f_n(x_n) \leq f(x_0)$ for \textbf{some} sequence $x_n$ with $x_n \stackrel d\to x_0$,
\end{enumerate}
and we write $f = \Glim_{n \to \infty} f_n$.
\end{definition}

The two inequality conditions are referred to as the \textit{liminf inequality} and \textit{limsup inequality} respectively.

The following properties are proven in Chapter 1 of \cite{braides}.
\begin{itemize}
    \item The Gamma limit, if it exists, is unique pointwise.  Moreover, it is given precisely by
    $$\left(\Glim_{n \to \infty} f_n\right)(x) = \inf\left\{\liminf_{n \to \infty} f_n(x_n) : x_n \stackrel d\to x\right\}.$$
    \item It is not necessarily true that the Gamma limit coincides with the pointwise limit, if both exist.
    \item If $f = \Glim f_n$ exists and $x_n$ is a minimizer for $f_n$, then every accumulation point of the sequence $\{x_n\}_n$ is a minimizer for $f$.
\end{itemize}

As an example, let us define $\Omega := (0,1)$ and consider the functional $F_\epsilon:W^{1,2}(\Omega) \to \R$ defined by
$$F_\epsilon(u) := \begin{cases}\int_0^1 (u^2-1)^2+\epsilon^2|u'|^2 \,dx, & \int_\Omega u\,dx = 0 \\ +\infty, & \text{otherwise}\end{cases}$$
for $\epsilon > 0$.  Although it is quite non-trivial, it can be shown that $F_\epsilon$ has a minimizer.

The $\epsilon^2|u'|^2$ term is a complication.  Can it be removed without much change to the minimizers?  As before, let $\epsilon_n \to 0^+$ be a sequence.  We can show that
$$\Glim_{n \to \infty} F_{\epsilon_n}(u) = \int_0^1 (u^2-1)^2\,dx$$
under $W^{1,2}(\Omega)$ convergence.  Now let us attempt to repeat the previous logic.  If we take $u_n$ to be a minimizer of $F_{\epsilon_n}$, then any subsequence of $u_n$ converging in $W^{1,2}(\Omega)$ must converge to a minimizer of $\int_0^1 (u^2-1)^2\,dx$.  However, we can show that 
$$\inf\left\{\int_0^1 (u^2-1)^2\,dx : \int_\Omega u\,dx = 0 \text{ and } u \in W^{2,1}(\Omega)\right\} = 0,$$ hence a minimizer $u_{\text{min}}$ would have to take the form $u_{\text{min}} = 1 \cdot 1_E + (-1) \cdot 1_{\Omega \setminus E}$ for some $E$ with $\leb^1(E) = \frac12$, which cannot be consistent with the requirement that $u \in W^{2,1}(\Omega)$.  Hence, there is no such minimizer!

Our failure to deduce anything meaningful from this Gamma convergence result is a consequence of missing a compactness result that guarantees the existence of converging subsequences of $u_n$.  Fortunately, without computing the exact minimizers, we can still recover a reasonable Gamma limit result by weakening the convergence.  Specifically, we may sacrifice the strong $W^{1,2}(\Omega)$ convergence for a weaker convergence called \textit{weak} $L^4(\Omega)$ \textit{convergence} (and written as ``$u_n \wc u$ in $L^4(\Omega)$").  The advantage of this weakening is that if we take $u_n$ to be a minimizer of $F_{\epsilon_n}$, then we can use \textit{weak compactness} to prove that there is a subsequence $u_{n_k}$ and some $u \in L^4(\Omega)$ for which $u_{n_k} \wc u$ in $L^4(\Omega)$, which is precisely the sort of compactness result we seek.  Under weak $L^4(\Omega)$ convergence, the Gamma limit changes to 
$$\Glim_{n \to \infty} F_{\epsilon_n}(u) = \int_0^1 f^{**}(u)\,dx,$$
where $f^{**}(z) := \begin{cases}(z^2-1)^2, & |z| > 1 \\ 0, & |z| \leq 1\end{cases}$ denotes the \textit{convex envelope} of $f(z) := (z^2-1)^2$.  Due to the existence of a compactness result, we are guaranteed the existence of a subsequence $u_{n_k}$ converging weakly in $L^2(\Omega)$ to a minimizer of $\int_\Omega f^{**}(u)\,dx$, and indeed there are many minimizers of this functional.


This demonstrates the importance of choosing the correct metric of convergence for the Gamma limit.  If the convergence is too strong, there may be no compactness result.  On the other hand, the weaker the convergence, the weaker the result we end up proving.  


\section{The Second-Order Singularly-Perturbed Problem}

Our first goal is to prove Theorem \ref{compactness2}.  As a steppingstone, we first prove a similar compactness result in the context of a first order singularly-perturbed problem, which is of interest in itself.

\begin{theorem}[Compactness for First Order Problem]\label{compactness1}
Let $\Omega = (a,b)$ be a non-empty open interval, and let $1 \leq p < \infty$.  Let $\epsilon_n$ be a sequence of positive reals with $\epsilon_n \to 0^+$.  
\begin{enumerate}[label=(\roman*)]
    \item If we have a sequence $u_n \in W^{1,p}(\Omega)$ with 
    $$C := \sup_{n \in \N} \int_\Omega \epsilon_n^{-1}(u_n^2-1)^2 + \epsilon_n|u_n'|^2\,dx < \infty,$$ then there exists a subsequence $u_{n_k}$ for which $u_{n_k} \to u$ in $L^p(\Omega)$ for some $u \in BPV(\Omega;\{-1,1\})$.
    \item Moreover, such a $u$ must satisfy $\essVar_\Omega u \leq \frac34C.$
\end{enumerate}
\end{theorem}

\begin{proof}
Since $\int_\Omega (u_n^2-1)^2\,dx \leq C\epsilon_n \to 0$, we see that $u_n^2-1 \to 0$ in $L^2(\Omega)$.  So, by extraction of a subsequence, we may assume that $u_n^2-1 \to 0$ almost everywhere.  In particular, we get that $|u_n| \to 1$ almost everywhere.

Applying the AM-GM inequality inside the integral gives
$$C \geq \int_\Omega \epsilon_n^{-1}(u_n^2-1)^2+\epsilon_n|u_n'|^2\,dx \geq \int_\Omega 2|u_n^2-1| \cdot |u_n'|\,dx = \int_\Omega |U_n'|\,dx$$
for all $n$, where $U_n := 2u_n-\frac23u_n^3$.  Taking $\tilde{U_n}$ to be the absolutely continuous representative of $U_n$, we deduce that $\Var_\Omega \tilde{U}_n \leq C$ for all $n$. Moreover, since $|u_n| \to 1$ almost everywhere, we have that $|U_n| \to 2-\frac23 = \frac43$ almost everywhere by continuity of $z \mapsto 2z-\frac23z^3$, and so in particular there is some $x \in \Omega$ for which $\tilde{U}_n(x)$ converges.  From this, we may now apply Helly's Selection Theorem to find a subsequence $\tilde{U}_{n_k}$ that converges pointwise to some $\tilde{U} \in BPV(\Omega)$, and $C \geq \liminf_{n \to \infty} \Var_\Omega \tilde{U}_n \geq \Var \tilde{U}$ by Theorem \ref{hellybound}.

It follows that $U_{n_k}$ converges to some $U$ almost everywhere that has representative $\tilde{U}$, with $U = \pm \frac43$ almost everywhere, and moreover $\essVar_\Omega U \leq \Var \tilde{U} \leq C < \infty$.  Hence, by taking $u = \frac34U$, we have that 
\begin{equation}
    \essVar_\Omega u \leq \frac34 C < \infty, \label{eq:essVarbound}
\end{equation} 
and particularly $u$ has a representative $\tilde{u} \in BPV(\Omega;\pm 1)$.

To see that $u_{n_k} \to u$ almost everywhere, recall that $|u_{n_k}| \to 1$ almost everywhere.  Taking representatives $\tilde{u}_{n_k}$, we have that almost every $x_0 \in \Omega$ satisfies $|\tilde{u}_{n_k}(x_0)| \to 1$, $2\tilde{u}_{n_k}-\frac23\tilde{u}_{n_k}^3 = \tilde{U}_{n_k}$ for all $k$, $\tilde{U}_{n_k}(x_0) \to \tilde{U}(x_0)$, and $\tilde{U}(x_0) = \frac43\tilde{u}(x_0)$.

For every such $x_0$ we have that $2\tilde{u}_{n_k}(x_0)-\frac23u_{n_k}(x_0)^3 \to \frac43\tilde{u}(x_0)$.  If $\tilde{u}(x_0) = 1$, then $2\tilde{u}_{n_k}(x_0)-\frac23\tilde{u}_{n_k}(x_0)^3 - \frac43 \to 0$, and by factoring we obtain $(\tilde{u}_{n_k}(x_0)-1)^2(\tilde{u}_{n_k}(x_0)+2) \to 0$.  But $|\tilde{u}_{n_k}(x_0)| \to 1$, so $\liminf_{k \to \infty} |\tilde{u}_{n_k}(x_0)+2| \geq 1$, and hence it must follow that $(\tilde{u}_{n_k}(x_0)-1)^2 \to 0$, i.e. $\tilde{u}_{n_k}(x_0) \to 1 = \tilde{u}(x_0)$.  By a symmetrical argument we see that if $\tilde{u}(x_0) = -1$ then $\tilde{u}_{n_k}(x_0) \to -1 = \tilde{u}(x_0)$.  This shows that $u_{n_k} \to u$ almost everywhere.

Lastly, to obtain $L^p(\Omega)$ convergence, we may  use the fact that $\essVar_\Omega U_{n_k} \leq C$ and that $U_{n_k}$ converges almost everywhere to deduce that $U_{n_k}$ is uniformly bounded by some constant $C'$.  Then the inverse image of $[-C,C]$ under $z \mapsto 2z-\frac23z^3$ is compact by continuity, thus $|U_{n_k}| = |2u_{n_k}-\frac23u_{n_k}^3| \leq C'$ implies that $|u_{n_k}| \leq C''$ for some constant $C''$.  In particular, we see that $u_{n_k}$ is uniformly bounded, so we may apply dominated convergence to obtain the desired $L^p(\Omega)$ convergence.  This proves item (i), and the bound claimed in item (ii) was acquired in \eqref{eq:essVarbound}.
\end{proof}

Next, we will prove an essential interpolation inequality.

\begin{theorem}\label{inter3}
Let $\Omega = (a,b)$ be a non-empty open interval, let $\Omega_0 = (a_0,b_0) \Subset (a,b)$, and let $1 \leq p \leq 4$.  Let $\epsilon_n$ be a sequence of positive reals with $\epsilon_n \to 0^+$.  If we have a sequence $u_n \in W^{2,p}(\Omega)$ with 
$$\sup_{n \in \N} \int_\Omega \epsilon_n^{-1}(u_n^2-1)^2 + \epsilon_n^3|u_n''|^2\,dx < \infty,$$ 
then for a universal constant $C' > 0$ we have that
$$\int_{\Omega_0} \epsilon_n|u_n'|^2\,dx \leq C'\int_\Omega \epsilon_n^{-1}(u_n^2-1)^2 + \epsilon_n^3|u_n''|^2\,dx$$
for all large enough $n$.
\end{theorem}

We begin by proving two lemmas.

\begin{lemma}\label{unif}
Let $\Omega$ be an open interval, $1 \leq p \leq 4$, and $u_n \in L^p(\Omega)$ such that $\lim_{n \to \infty} \int_a^b (u_n^2-1)^2\,dx = 0$.  Then $\{|u_n|^p\}_n$ is uniformly integrable.
\end{lemma}

\begin{proof}
First, find some $c > 0$ so small that $cz^p \leq (z^2-1)^2$ for all $z$ large enough; say, $z \geq M_0$.  This is possible by virtue of order of growth, since $1 \leq p \leq 4$.  Now 
$$\int_\Omega |u_n|^p\,dx = \int_{\{|u_n| < M_0\}} |u_n|^p\,dx + \int_{\{|u_n| \geq M_0\}} |u_n|^p\,dx \leq M_0^p + \int_{\{|u_n| \geq M_0\}} \frac1c(u_n^2-1)^2\,dx < \infty,$$
so $|u_n|^p$ is integrable for all $n$.

Fix $\eta > 0$.  Find $N_\eta$ such that $\int_\Omega (u_n^2-1)^2\,dx \leq c\eta$ for all $n > N_\eta$.  Lastly, by integrability of $|u_n|^p$ for $n=1,2,\cdots,N_\eta$, we choose $M_\eta \geq M_0$ so large that $\int_{|u_n| \geq M_\eta} |u_n|^p\,dx < \eta$ for all such $n$.  This $M_\eta$ witnesses the uniform integrability since for all $n \geq N_\eta$ we have
$$\int_{|u_n| \geq M_\eta} |u_n|^p \,dx \leq \frac1c\int_{|u_n| \geq M_\eta} (u_n^2-1)^2 \,dx \leq \eta.$$
\end{proof}

\begin{lemma}\label{count}
Let $\Omega$ be an open interval.  Then for every $u \in W^{2,p}(\Omega)$, the set $Z = \{x \in \Omega : u(x) = 0, u'(x) \neq 0\}$ is countable.
\end{lemma}

\begin{proof}
Let $x_0 \in Z$.  It suffices to show that $x_0$ is an isolated point of $Z$.  We may assume without loss of generality that $u'(x_0) > 0$.

Since $u \in W^{2,p}(\Omega)$ we know that $u'$ is continuous (See Theorem \ref{ACLcondition}).  Thus, there exists $\delta > 0$ such that $u'(x) > 0$ for all $x \in (x_0-\delta,x_0+\delta)$.  We conclude that there is no $x \in (x_0-\delta,x_0+\delta)$ for which $u(x) = 0$ (so that in particular, $x \not\in Z$), otherwise we obtain a contradiction from Rolle's Theorem.
\end{proof}

Although it is not worth stating as a lemma, we will be using the inequality $(|z|-1)^2 \leq (z^2-1)^2$ quite liberally, which follows from factoring $(z^2-1)^2$ and taking cases on the sign of $z$.

We may now prove Theorem \ref{compactness2}.  
\begin{proof}
Let 
\begin{equation}
C = \sup_{n \in \N} \int_{a}^{b} \epsilon_n^{-1}(u_n^2-1)^2 + \epsilon_n^3|u_n''|^2\,dx. \label{eq:defineC}    
\end{equation}
Since $(a_0,b_0) \Subset (a,b)$ and $\epsilon_n \to 0$, it is sufficient to obtain an interpolation inequality of the form
\begin{equation}
    \int_{\Omega_{\epsilon_n}} \epsilon_n|u_n'|^2\,dx \leq C'\int_\Omega \epsilon_n^{-1}(u_n^2-1)^2 + \epsilon_n^3|u_n''|^2\,dx, \label{eq:desiredform}
\end{equation}
where $C'$ is a universal constant and $\Omega_\epsilon := (a+\epsilon,b-\epsilon)$ for all $\epsilon > 0$.

Note that if $\int_\Omega \epsilon_n^{-1}(u_n^2-1)^2\,dx = 0$, then we must have either $u_n \equiv 1$ or $u_n \equiv -1$.  Hence $\int_{\Omega_{\epsilon_n}} \epsilon_n|u_n'|^2\,dx = 0$, so the inequality \eqref{eq:desiredform} holds and there is nothing more to show.  Thus, we may assume the contrary, so that
\begin{equation}
    \int_\Omega \epsilon_n^{-1}(u_n^2-1)^2 + \epsilon_n^3|u_n''|^2\,dx > 0. \label{eq:dumbcase}
\end{equation}

It is clear from \eqref{eq:defineC} that
\begin{equation}
  \int_\Omega |u_n''|^2\,dx \leq \frac{C}{\epsilon_n^3} \label{eq:W22bound}
\end{equation}
and
$$\int_\Omega (u_n^2-1)^2\,dx \leq \epsilon_nC$$
for all $n$. 

We divide our analysis of $u_n$ into three sets:  Where $u_n = 0$, where $u_n$ is \textit{near} 0, and where $u_n$ is \textit{far} from 0.  

Firstly, the argument in Lemma \ref{unif} implies that $u_n \in L^2(\Omega)$, and \eqref{eq:W22bound} implies that $u_n'' \in L^2(\Omega)$.  It follows by Theorem \ref{inter2} that $u_n' \in L^2(\Omega)$.  Hence, by this fact and the assumption \eqref{eq:dumbcase}, there must exist $\alpha_n \in (0,\tfrac12)$ so small that 
\begin{equation}
    \int_{0 < |u_n| \leq \alpha_n} |u_n'|^2\,dx \leq \int_\Omega \epsilon_n^{-1}(u_n^2-1)^2 + \epsilon_n^3|u_n''|^2\,dx. \label{closebound}
\end{equation}
Our three sets in question are then $u_n^{-1}(\{0\})$, $u_n^{-1}\big([-\alpha_n,0) \cup (0,\alpha_n]\big)$, and $u_n^{-1}\big((-\infty,-\alpha_n) \cup (\alpha_n,+\infty)\big)$.

For the set $u_n^{-1}(\{0\})$, we claim that 
\begin{equation}
    \int_{u_n = 0} |u_n'|\,dx = 0. \label{eq:zerobound}
\end{equation}
To prove this, it suffices to show that
$$\int_{\{y : u_n(y) = 0, u_n'(y) \neq 0\}} |u_n'|\,dx = 0$$
since the integral over the set where $u_n' = 0$ is clearly 0.  But by Lemma \ref{count}, the domain over which we integrate here is countable, so the integral is indeed 0.
    
The set $u_n^{-1}\big([-\alpha_n,0) \cup (0,\alpha_n]\big)$ is already handled by \eqref{closebound}.
    
It remains to handle $u_n^{-1}((-\infty,-\alpha_n) \cup (\alpha_n,+\infty))$.  This set is open by continuity, so it is the at-most countable union of disjoint open intervals.  Let those intervals that do not have either $a$ or $b$ as an endpoint be enumerated as $\{(a_i,b_i)\}_{i \in I}$ for a countable, possibly empty index set $I$.  Let the union of these intervals be $E_n$, so $E_n$ is where $|u_n| > \alpha_n$ except possibly near the endpoints of $(0,1).$

Then for all $i \in I$ we have that $u_n(a_i) = u_n(b_i)$ by a continuity argument.  Thus by Rolle's Theorem there is $c_i \in (a_i,b_i)$ for which $u_n'(c_i) = 0$.  Moreover, $u_n \neq 0$ over $(a_i,b_i)$, thus $v_n := |u_n|-1$ is differentiable over $(a_i,b_i)$.  If we assume without loss of generality that $u_n > 0$ over $(a_i,b_i)$, then in fact $v_n = u_n-1$, so it is clear that $v_n \in W^{2,p}(a_i,b_i)$, with $v_n' = u_n'$ and $v_n'' = u_n''$.  By Lemma \ref{inter1}, it follows that
\begin{align*}
  \int_{a_i}^{b_i} |u_n'|^2\,dx &\leq c_1\left(\int_{a_i}^{b_i} (|u_n|-1)^2\,dx\right)^{1/2}\left(\int_{a_i}^{b_i} |u_n''|^2\,dx\right)^{1/2}  \\
  &\leq c_1\left(\int_{a_i}^{b_i} (u_n^2-1)^2\,dx\right)^{1/2}\left(\int_{a_i}^{b_i} |u_n''|^2\,dx\right)^{1/2}
\end{align*}
for a constant $c_1 > 0$.  Multiplying by $\epsilon_n$ gives
$$\int_{a_i}^{b_i} \epsilon_n|u_n'|^2\,dx \leq c_1\left(\int_{a_i}^{b_i} \epsilon_n^{-1}(u_n^2-1)^2\,dx\right)^{1/2}\left(\int_{a_i}^{b_i} \epsilon_n^3|u_n''|^2\,dx\right)^{1/2},$$
then applying the AM-GM inequality gives
$$\int_{a_i}^{b_i} \epsilon_n|u_n'|^2\,dx \leq \frac{c_1}{2}\int_{a_i}^{b_i} \epsilon_n^{-1}(u_n^2-1)^2 + \epsilon_n^3|u_n''|^2\,dx.$$
Summing this inequality over all $i \in I$, we conclude that
\begin{equation}
    \int_{E_n} \epsilon_n|u_n'|^2\,dx \leq \frac{c_1}{2}\int_{E_n} \epsilon_n^{-1}(u_n^2-1)^2 + \epsilon_n^3|u_n''|^2\,dx. \label{mostbound}
\end{equation}
We are almost done.  Let $S_n = \sup\{x \in \Omega : |u_n(y)| \geq \alpha_n\,\forall a < y \leq x\}$ and $T_n = \inf\{x \in \Omega : |u_n(y)| \geq \alpha_n\,\forall x \leq y < b\}$.  If $\{x \in \Omega : |u_n(y)| \geq \alpha_n\,\forall a < y \leq x\}$ is empty then we take $S_n = a$, and similarly for $T_n$.  The intervals $(a,S_n)$ and $(T_n, b)$, if distinct and non-empty, are the two intervals that we omitted from $\{x \in \Omega : u_n(x) > \alpha_n\}$ to obtain $E_n$, so that $E_n \cup (a,S_n) \cup (T_n,b) = \{x \in \Omega : u_n(x) > \alpha_n\}$.  It remains to handle these two intervals.


If $S_n - a \leq \epsilon_n$, then clearly $\int_{(a,S_n) \cap \Omega_{\epsilon_n}} \epsilon_n|u_n'|^2\,dx = 0$.  If otherwise $S_n - a > \epsilon_n$ then we may apply Theorem \ref{inter2} to $|u_n|-1 \in W^{2,p}(a,S_n)$ to get
\begin{align*}
    \int_a^{S_n} |u_n'|^2\,dx &\leq c_2\epsilon_n^{-2}\int_a^{S_n} (|u_n|-1)^2\,dx + c_1\epsilon_n^2\int_a^{S_n} |u_n''|^2\,dx \\
    &\leq c_2\epsilon_n^{-2}\int_a^{S_n} (u_n^2-1)^2\,dx + c_1\epsilon_n^2\int_a^{S_n} |u_n''|^2\,dx,
\end{align*}
for a constant $c_2 > 0$.  Multiplying by $\epsilon_n$ gives
\begin{equation}
\int_{(a,S_n) \cap \Omega_{\epsilon_n}} \epsilon_n|u_n'|^2\,dx \leq \int_a^{S_n} \epsilon_n|u_n'|^2\,dx \leq c_2\int_a^{S_n} \epsilon_n^{-1}(u_n^2-1)^2 + \epsilon_n^3|u_n''|^2\,dx. \label{eq:aendbound}
\end{equation}
Thus the inequality \eqref{eq:aendbound} holds in either the cases $S_n - a \leq \epsilon_n$ and $S_n - a > \epsilon_n$.  Applying the same arguments above to the interval $(T_n,b)$, we see that in either of the cases $b - T_n \leq \epsilon_n$ and $b - T_n > \epsilon_n$ we may obtain the inequality
\begin{equation}
\int_{(T_n,b) \cap \Omega_{\epsilon_n}} \epsilon_n|u_n'|^2\,dx \leq c_2\int_{T_n}^{b} \epsilon_n^{-1}(u_n^2-1)^2 + \epsilon_n^3|u_n''|^2\,dx. \label{eq:bendbound}
\end{equation}
By summing \eqref{mostbound}, \eqref{eq:aendbound}, and \eqref{eq:bendbound}, we get the bound
\begin{align}
    \int_{\{|u_n| > \alpha_n\} \cap \Omega_{\epsilon_n}} \epsilon_n|u_n'|^2\,dx &\leq \int_{E_n \cup ((a,S_n) \cap \Omega_{\epsilon_n}) \cup ((T_n,b) \cap \Omega_{\epsilon_n})} \epsilon_n|u_n'|^2\,dx \nonumber \\ 
     &\leq \max(c_1/2,c_2)\int_{E_n \cup (a,S_n) \cup (T_n,b)} \epsilon_n^{-1}(u_n^2-1)^2 + \epsilon_n^3|u_n''|^2\,dx  \nonumber \\
     &= \max(c_1/2,c_2)\int_{|u_n| > \alpha_n} \epsilon_n^{-1}(u_n^2-1)^2 + \epsilon_n^3|u_n''|^2\,dx. \label{eq:endbound}
\end{align}

At last, once we add up \eqref{closebound}, \eqref{eq:zerobound}, and \eqref{eq:endbound}, we can arrive at the inequality
$$\int_{\Omega_{\epsilon_n}} \epsilon_n|u_n'|^2\,dx \leq c_3\int_{\Omega} \epsilon_n^{-1}(u_n^2-1)^2 + \epsilon_n^3|u_n''|^2\,dx,$$
where $c_3 = 1 + \max(c_1/2,c_2)$.  This is of the form \eqref{eq:desiredform}, so we are done.
\end{proof}

We now turn to the proof of Theorem \ref{compactness2}.

\begin{proof}
Let $C = \sup_{n \in \N} F_{\epsilon_n}(u_n)$.  We first claim that for each $\delta > 0$ there exists a subsequence $u_{n_k}$ of $u_n$ such that $u_{n_k} \to v$ in $L^p(a+\delta,b-\delta)$ for some function $v \in BPV((a+\delta,b-\delta); \pm 1)$, with $\essVar_{(a+\delta,b-\delta)} v \leq C_3$ for a constant $C_3$ not depending on $\delta$.

To see this, we apply Theorem \ref{inter3} to the interval $(a+\delta,b-\delta)$ to obtain the bound
$$\int_{a+\delta}^{b-\delta} \epsilon_n|u_n'|^2\,dx \leq \int_\Omega \epsilon_n^{-1}(u_n^2-1)^2 + \epsilon_n^3|u_n''|^2\,dx \leq C$$
for all large enough $n$.  It follows that
$$\int_{a+\delta}^{b-\delta} \epsilon_n^{-1}(u_n^2-1)^2 + \epsilon_n|u_n'|^2\,dx \leq 2C$$
for all such $n$, and so by item (i) of Theorem \ref{compactness1} applied to the interval $(a+\delta,b-\delta)$, we may find a subsequence $u_{n_k}$ with $u_{n_k} \to v$ in $L^p(a+\delta,b-\delta)$ for some $v \in BPV((a+\delta,b-\delta); \pm 1)$, and moreover by item (ii) of Theorem \ref{compactness1} we must have that $\essVar_{(a+\delta,b-\delta)} v \leq \frac32C$.  Hence the claim is proven with $C_3 = \frac32C$.

We finish with a diagonalization argument.  Let $\Omega_m = \left(a+\delta_m,b-\delta_m\right)$ for a sequence $\delta_m \to 0^+$.  For each $m$ we define a subsequence $u_{m,n}$ of $u_n$ recursively as follows:
\begin{itemize}
    \item $u_{1,n} = u_n$ for all $n$.
    \item For $m \geq 2$, we take $\{u_{m,n}\}_n$ to be a subsequence of $\{u_{m-1,n}\}_n$ that converges in $L^p(\Omega_m)$ to some $v_m:\Omega_m \to \{-1,1\}$ with $\essVar_{\Omega_m} v_m \leq C_3$, which exists by the claim.
\end{itemize}

Let $w_n = u_{n,n}$.  We claim that for some $u \in L^p(\Omega;\{-1,1\})$ we have that $w_n \to u$ in $L^p(\Omega_m)$ for every $m$, so that in particular we have $w_n \to u$ in $L^p_{\text{loc}}(\Omega)$. 

To see this, consider $m_1,m_2$ with $m_2 > m_1$, so that $\{u_{m_2,n}\}_n$ is a subsequence of $\{u_{m_1,n}\}_n$.  Then we may write $u_{m_2,k} = u_{m_1,n_k}$ for some sequence $n_k$, and so
$$\|v_{m_2}-v_{m_1}\|_{L^p(\Omega_{m_1})} \leq \|v_{m_2}-u_{m_2,k}\|_{L^p(\Omega_{m_1})}+\|u_{m_1,n_k}-v_{m_1}\|_{L^p(\Omega_{m_1})}.$$
Sending $k \to +\infty$ we deduce that $v_{m_2} = v_{m_1}$ almost everywhere. over $\Omega_{m_1}$, so in general $v_{m_2}$ is an extension of $v_{m_1}$ for all $m_2$ and $m_1$ with $m_2 > m_1$.  Take $u$ to be the maximal such extension, defined over all of $\Omega$.  Then for any $m$, we have that
$$\|w_n - u\|_{L^p(\Omega_m)} = \|u_{n,n}-v_m\|_{L^p(\Omega_m)} \xrightarrow{n \to +\infty} 0$$
because $\{u_{n,n}\}_{n=m}^\infty$ is a subsequence of $\{u_{m,n}\}_n$, which converges to $v_m$ in $L^p(\Omega_m)$.  

This proves the claim.  In particular, we see that $w_n \to u$ in measure.  Moreover the family $\{|w_n|^p\}_n$ is uniformly integrable by Lemma \ref{unif}.  Thus by Vitali, we get that in fact, $w_n \to u$ in $L^p(\Omega)$.

Lastly, to see that $u \in BPV(\Omega)$, note that 
$$\essVar_\Omega u = \lim_{m \to \infty} \essVar_{\Omega_m} u = \lim_{m \to \infty} \essVar_{\Omega_m} v_m \leq C_3,$$
so we are done.
\end{proof}

We now find the Gamma-limit under $L^p(\Omega)$ convergence.  

\begin{theorem}\label{gamma2}
Let $\Omega = (a,b)$ be a non-empty open interval, and let $1 \leq p < \infty$.  For each $\epsilon > 0$, define a functional $F_\epsilon:L^p(\Omega) \to \overline{\R}$ as in \eqref{eq:Fdef}. Let $\epsilon_n$ be a sequence of positive reals with $\epsilon_n \to 0^+$.  For all $u \in L^p(\Omega)$, we have
$$\Glim_{n \to +\infty} F_{\epsilon_n}(u) = \begin{cases}\alpha\essVar_\Omega u, & u \in BPV(\Omega;\pm 1) \\ +\infty, & \text{otherwise}\end{cases}$$
under $L^p(\Omega)$ convergence, where $\alpha$ is defined as in \eqref{eq:alpha}.
\end{theorem}

We begin with a useful interpolation result.

\begin{lemma}\label{glue}
Let $T > 0$ and $A,m \in \R$ with $m \neq 0$.  Then there exists $f:[0,T] \to \R$ satisfying the following properties:
\begin{enumerate}[label=(\roman*)]
    \item $f \in C^\infty([0,T])$,
    \item $f(0) = A$, $f'(0) = m$, $f(T) = 0$, and $f'(T) = 0$,
    \item $\int_0^T |f^{(k)}(x)|^2\,dx \leq C_k(A^2+m^2T^2)T^{1-2k}$ for a constant $C_k$ depending only on $k$,
    \item $\sup_{[0,T]} |f| \leq |A| + \frac12|m| \cdot|T|$.
\end{enumerate}
\end{lemma}

\begin{proof}
We take $g:\R \to \R$ with
$$g(x) = \begin{cases}mx+A, & -T/2 < x < T/2 \\ 0, & \text{otherwise}\end{cases}.$$
Now let $\varphi \in C^\infty_c((-1,1))$ be a non-negative and symmetric mollifier, so that $\varphi(x) = \varphi(-x)$ for all $x \in (-1,1)$.  For each $\epsilon > 0$ we define $\varphi_\epsilon(x) := \frac1\epsilon\varphi\left(\frac x\epsilon\right)$ and the mollification
$$g_\epsilon(x) := \int_{x-\epsilon}^{x+\epsilon} g(y)\varphi_\epsilon(x-y)\,dy.$$
For $\epsilon < T/2$ we recover the properties $g_\epsilon(0) = A$, $g'_\epsilon(0) = m$, $g_\epsilon(T) = 0$, and $g_\epsilon'(T) = 0$, because by symmetry of $\varphi_\epsilon$ we have that $g_\epsilon(x) = mx+A$ for $x$ near $0$ and $g_\epsilon(x) = 0$ for $x$ near $T$.  
Note that for each $k$, we have that $|g(y)\varphi_\epsilon^{(k)}(x-y)| \leq \epsilon^{-k}\|\varphi^{(k)}\|_\infty|g(y)|$ for all $x,y \in \R$.  As $g$ is integrable, we may use dominated convergence to obtain
\begin{align*}
  g^{(k)}_\epsilon(x) &= \d[k]{}{x}\int_{x-\epsilon}^{x+\epsilon} g(y)\varphi_\epsilon(x-y)\,dy = \d[k]{}{x}\int_\R g(y)\varphi_\epsilon(x-y)\,dy \\
  &= \int_\R g(y)\varphi_\epsilon^{(k)}(x-y)\,dy = \int_{x-\epsilon}^{x+\epsilon} g(y)\varphi_\epsilon^{(k)}(x-y)\,dy.
\end{align*}
Now we integrate in $x$ and apply Minkowski's inequality for integrals to get
\begin{align*}
\int_0^T g_\epsilon^{(k)}(x)^2\,dx &= \int_0^T\left(\int_{x-\epsilon}^{x+\epsilon} g(y)\varphi_\epsilon^{(k)}(x-y)\,dy\right)^2\,dx \\
&\leq \left(\int_{-\epsilon}^{T+\epsilon}\left(\int_{(y-\epsilon,y+\epsilon) \cap [0,T]}|g(y)|^2|\varphi_\epsilon^{(k)}(x-y)|^2\,dx\right)^{1/2}\,dy\right)^2 \\
&\leq \left(\int_{-\epsilon}^{T+\epsilon}|g(y)|\left(\int_{y-\epsilon}^{y+\epsilon}|\varphi_\epsilon^{(k)}(x-y)|^2\,dx\right)^{1/2}\,dy\right)^2\\
&= \left(\int_{-\epsilon}^{T+\epsilon}|g(y)|\,dy\right)^2\int_{-\epsilon}^{\epsilon}|\varphi_\epsilon^{(k)}(z)|^2\,dz.
\end{align*}
Handling the first integral is a matter of computation, which gets us the bound
$$\int_{-\epsilon}^{T+\epsilon}|g(y)|\,dy \leq \int_{-\epsilon}^{T+\epsilon} |m| \cdot |y| + |A|\,dy = \frac{|m|}{2} \cdot \left((T+\epsilon)^2 + \epsilon^2\right) + (T+2\epsilon)|A|.$$
As for the second integral, we may write $\varphi_\epsilon^{(k)}(x) = \d[k]{}{x}\frac{1}{\epsilon}\varphi\left(\frac{x}{\epsilon}\right) = \frac{1}{\epsilon^{k+1}}\varphi^{(k)}\left(\frac{x}{\epsilon}\right)$, so that
$$\int_{-\epsilon}^{\epsilon}|\varphi_\epsilon^{(k)}(z)|^2\,dz = \int_{-\epsilon}^{\epsilon}\frac{1}{\epsilon^{2k+2}}|\varphi^{(k)}(z/\epsilon)|^2\,dz = \frac{1}{\epsilon^{2k+1}}\int_{-1}^{1}|\varphi^{(k)}(z)|^2\,dz.$$
Altogether, we get that
$$\int_0^T g_\epsilon^{(k)}(x)^2\,dx \leq \frac{\left(\frac{|m|}{2} \cdot \left((T+\epsilon)^2 + \epsilon^2\right) + (T+2\epsilon)|A|\right)^2}{\epsilon^{2k+1}} \cdot \|\varphi^{(k)}\|_{L^2(\R)}^2.$$
Choosing $\epsilon = T/4$, and applying the inequality $(a+b)^2 \leq 2a^2+2b^2$, we see that
$$\int_0^T g_{T/4}^{(k)}(x)^2\,dx \leq \frac{|m|^2T^4+|A|^2T^2}{T^{2k+1}} \cdot C_k = C_k(|A|^2+|m|^2T^2)T^{1-2k}$$ 
for a constant $C_k$ depending only on $k$.

Lastly, since $|g(x)| \leq |m|\cdot|x| + |A| \leq |m| \cdot \frac{|T|}{2} + |A|$ for all $x$, we must have 
$$|g_{T/4}(x)| \leq \int_{x-\epsilon}^{x+\epsilon}|g(y)|\varphi_\epsilon(x-y)\,dy \leq \left(|m| \cdot \frac{|T|}{2} + |A|\right)\int_{x-\epsilon}^{x+\epsilon}\varphi_\epsilon(x-y)\,dy = |m| \cdot \frac{|T|}{2} + |A|$$
for all $x$, so that in particular $\sup_{[0,T]}|g_{T/4}(x)| \leq |m| \cdot \frac{|T|}{2} + |A|$.  We conclude that taking $f \equiv g_{T/4}$ satisfies the required properties.
\end{proof}

Now we prove the theorem.

\begin{proof} $ $\newline

\noindent\textbf{Step 1}

If $u \not\in L^p(\Omega;\pm 1)$ then for some $\epsilon > 0$, the set $E = \{x \in \Omega : ||u(x)|-1| \geq \epsilon\}$ has positive measure.  We now claim that the inequality $(z^2-1)^2 \geq \epsilon^{2-p}||z|-1|^p$ holds for all $z$ with $||z|-1| \geq \epsilon$.  To see this, note that for all such $z$ with $z \geq 0$ we have
$$|z+1|^2 \cdot |z-1|^{2-p} \geq \epsilon^{2-p}$$
which rearranges to the claimed inequality, and similarly for $z \leq 0$.  Hence we may use this and Minkowski's inequality to write
\begin{align*}
\int_E \epsilon_n^{-1}(u_n^2-1)^2\,dx &\geq \int_E \epsilon_n^{-1}\epsilon^{2-p}||u_n|-1|^p\,dx \\
&\geq \epsilon_n^{-1}\epsilon^{2-p}\left(\left(\int_E ||u|-1|^p\,dx\right)^{1/p} - \left(\int_E \big||u|-|u_n|\big|^p\,dx\right)^{1/p}\right)^p \\
&\geq \epsilon_n^{-1}\epsilon^{2-p}\left(\epsilon\leb^1(E)^{1/p} - \left(\int_E |u-u_n|^p\,dx\right)^{1/p}\right)^p.
\end{align*}
Now it is clear that $F_{\epsilon_n}(u_n) \to +\infty$ because $\|u-u_n\|_{L^p(E)} \to 0$ and $\epsilon_n^{-1} \to +\infty$.  Hence $\Glim_{n \to \infty} F_{\epsilon_n}(u) = +\infty$.

\noindent\textbf{Step 2}

We now consider the case $u \in BPV(\Omega;\pm 1)$.  

Suppose $u$ jumps between $-1$ and $1$ or vice versa at the  $J$  points $x_1,x_2,\cdots,x_J \in \Omega$.  Find $\delta_1,\cdots,\delta_J > 0$ so small that the intervals $(a_i,b_i)$ are disjoint, where $(a_i,b_i) := (x_i-\delta_i,x_i+\delta_i)$.


Let $u_n \to u$ in $L^p(\Omega)$.  Suppose $L = \ds\liminf_{n \to \infty} F_{\epsilon_n}(u_n)$.  For the liminf inequality, we need to prove that $L \geq \alpha\essVar u$.

By extraction of a subsequence, assume that actually $L = \dlim_{n \to \infty} F_{\epsilon_n}(u_n)$.  Fix an arbitrary $\eta \in (0,1)$.  Then we may find $N_\eta$ so large that
$$L + \eta \geq F_{\epsilon_n}(u_n)$$
for all $n \geq N_\eta$.

Pick some $x_i$ and, for simplicity, assume without loss of generality that $u$ jumps from $-1$ to $1$ at $x_i$.  

The key claim that we shall prove in this step is that, for all $n$ large enough, we may modify $u_n$ in $(a_i,b_i)$ by ``anchoring" it to $-1$ at $a_i$ and to $1$ at $b_i$ with a derivative of $0$ at both points, without spending more than $C\eta$ potential energy for a constant $C$ not depending on $\eta$.  More concretely, there exists $v_{n,i}:[a_i,b_i] \to \R$ such that:
\begin{itemize}
    \item $v_{n,i} \in W^{2,p}((a_i,b_i))$,
    \item $v_{n,i}(a_i) = -1$, $v_{n,i}(b_i) = 1$,
    \item $v_{n,i}'(a_i) = v_{n,i}'(b_i) = 0$,
    \item $\ds  \int_{a_i}^{b_i} \epsilon_n^{-1}(v_{n,i}^2-1)^2 + \epsilon_n^3|v_{n,i}''|^2\,dx \leq \int_{a_i}^{b_i} \epsilon_n^{-1}(u_n^2-1)^2 + \epsilon_n^3|u_n''|^2\,dx + C\eta$.
\end{itemize}

To do this, we consider only $n \geq N_\eta$, and it suffices to modify $u_n$ in $(x_i,b_i)$ by affixing to $1$ in the manner described.  First, observe that since $u_n \to u$ in $L^p(\Omega)$, we have in particular that $u_n \to 1$ in $L^p((x_i,b_i))$, so $\|u_n-1\|_{L^p((x_i,b_i))} \to 0$ as $n \to +\infty$.  Putting this aside, note that we also have $\int_{x_i}^{b_i} \frac{|u_n-1|^p}{\|u_n-1\|_{L^p((x_i,b_i))}^p}\,dx = 1$.  Combining this with $F_{\epsilon_n}(u_n) \leq L+\eta$, it follows that
$$\int_{x_i}^{b_i} \epsilon_n^{-1}(u_n^2-1)^2 + \epsilon_n^3|u_n''|^2 + \frac{|u_n-1|^p}{\|u_n-1\|_{L^p((x_i,b_i))}^p}\,dx \leq L+\eta+1.$$

Now let $K_n = \lceil \epsilon_n^{-1} \rceil$ and subdivide the interval $(x_i,b_i)$ into $K_n$ same-length intervals $(y_{j-1},y_j)$, so that:
$$x_i = y_0 < y_1 < \cdots < y_{K_n} = b_i$$
By a ``pigeonhole principle-like argument", there exists $1 \leq j \leq K_n$ such that
\begin{equation}
    \int_{y_{j-1}}^{y_j} \epsilon_n^{-1}(u_n^2-1)^2 + \epsilon_n^3|u_n''|^2 + \frac{|u_n-1|^p}{\|u_n-1\|_{L^p((x_i,b_i))}^p}\,dx \leq \frac{L+\eta+1}{K_n} \leq \epsilon_n(L+\eta+1). \label{eq:pigeonbound}
\end{equation}
A particular consequence is the bound $\int_{y_{j-1}}^{y_j} \epsilon_n^{-1}(u_n^2-1)^2 + \epsilon_n^3|u_n''|^2 \,dx \leq \frac{L+\eta+1}{K_n} \leq \epsilon_n(L+\eta+1)$, which implies by Theorem \ref{inter3} (applied with the intervals in the inclusion $(\frac23y_{j-1}+\frac13y_j,\frac13y_{j-1}+\frac23y_j) \Subset (y_{j-1},y_j)$) that
$$\int_{\frac23y_{j-1}+\frac13y_j}^{\frac13y_{j-1}+\frac23y_j} \epsilon_n|u_n'|^2\,dx \leq c\int_{y_{j-1}}^{y_j} \epsilon_n^{-1}(u_n^2-1)^2 + \epsilon_n^3|u_n''|^2\,dx \leq \epsilon_nc(L+\eta+1)$$
for all sufficiently large $n$, for a universal constant $c$.  Combining this with \eqref{eq:pigeonbound}, we obtain
$$\int_{\frac23y_{j-1}+\frac13y_j}^{\frac13y_{j-1}+\frac23y_j} \epsilon_n|u_n'|^2 + \frac{|u_n-1|^p}{\|u_n-1\|_{L^p((x_i,b_i))}^p}\,dx \leq \epsilon_n(c+1)(L+\eta+1).$$
Let the integrand $\epsilon_n|u_n'|^2 + \frac{|u_n-1|^p}{\|u_n-1\|_{L^p((x_i,b_i))}^p}$ be $H$.  If there does not exist $x \in (\frac23y_{j-1}+\frac13y_j,\frac13y_{j-1}+\frac23y_j)$ for which $H(x) \leq \frac{6K_n\epsilon_n(c+1)(L+\eta+1)}{b_i-x_i}$, then 
$$\epsilon_n(c+1)(L+\eta+1) \geq \int_{\frac23y_{j-1}+\frac13y_j}^{\frac13y_{j-1}+\frac23y_j} H\,dx \geq \frac{6K_n\epsilon_n(c+1)(L+\eta+1)}{b_i-x_i} \cdot \frac{y_j-y_{j-1}}{3}$$
$$=\frac{6K_n\epsilon_n(c+1)(L+\eta+1)}{b_i-x_i} \cdot \frac{b_i-x_i}{3K_n} = 2\epsilon_n(c+1)(L+\eta+1),$$
which is a contradiction.  We conclude that for all large enough $n$, there exists $x_0 \in (\frac23y_{j-1}+\frac13y_j,\frac13y_{j-1}+\frac23y_j)$, depending on $n$, such that $H(x_0) \leq (C_1/2)K_n\epsilon_n \leq (C_1/2) \cdot \frac{\epsilon_n+1}{\epsilon_n} \cdot \epsilon_n = (C_1/2)(\epsilon_n+1) \leq C_1$ for a constant $C_1$ that does not depend on $n$.  In particular, we now know that such an $x_0$ satisfies:
\begin{enumerate}
    \item $\epsilon_n^2|u_n'(x_0)|^2 \leq \epsilon_nC_1$,
    \item $|u_n(x_0)-1| \leq C_1^{1/p}\|u_n-1\|_{L^p((x_i,b_i))}$.
\end{enumerate}

These properties, combined with the fact that $u_n \to 1$ in $L^p(x_i,b_i)$, imply that, for all large enough $n$, we may obtain $\epsilon_n^2|u_n'(x_0)|^2 + |u_n(x_0)-1|^2 \leq \eta$.

Let $A = u_n(x_0)-1$, $m = u_n'(x_0)$, and $T = y_j - x_0$.  Using these constants, we may find a smooth $f:[0,T] \to \R$ as described in Lemma \ref{glue}.  We are now ready to define $v_{n,i}$ over $(x_i,b_i)$ as
$$v_{n,i}(x) := \begin{cases}u_n(x), & x_i < x \leq x_0 \\ f(x-x_0)+1, & x_0 < x \leq y_j \\ 1, & y_j < x < b_i\end{cases}.$$
By virtue of $f$ being a smooth connector, we must have $v_{n,i} \in W^{2,p}(x_i,b_i)$.  Moreover, by property (iii) in Lemma \ref{glue}, we have the bounds
\begin{equation}
    \int_{x_i}^{b_i} (v_{n,i}-1)^2\,dx = \int_{x_0}^{y_j} |f(x-x_0)|^2\,dx \leq C_0(A^2+m^2T^2)T \label{eq:vbound1}
\end{equation}
and
\begin{equation}
    \int_{x_i}^{b_i} |v_{n,i}''|^2\,dx = \int_{x_0}^{y_j} |f''(x-x_0)|^2\,dx \leq C_2(A^2+m^2T^2)T^{-3} \label{eq:vbound2}
\end{equation}
for universal constants $C_0,C_2 > 0$.  

Since \eqref{eq:vbound1} is not quite a bound on the integral of $(v_{n,i}^2-1)^2$, we will need to prove the inequality
\begin{equation}
  \int_{x_i}^{b_i} (v_{n,i}^2-1)^2\,dx \leq C'\int_{x_i}^{b_i} (v_{n,i}-1)^2\,dx \label{eq:transition}  
\end{equation}
for all $n$ large enough, for some constant $C' > 0$.  Indeed, observe that for all $x \in (x_i,b_i)$, we have by property (iv) of Lemma \ref{glue} that
$$|v_{n,i}(x)+1| \leq 2+|v_{n,i}-1| \leq 2+\sup_{[0,T]} |f| \leq 2+|A|+\frac12|m| \cdot |T|.$$
Now, $T \leq \epsilon_n$, and from $\epsilon_n^2|u_n'(x_0)|^2 + |u_n(x_0)-1|^2 \leq \eta \leq 1$ we have that $|A| = |u_n(x_0)-1| \leq 1$ and $|m| \cdot |T| = |u_n'(x_0)| \cdot (y_j-x_0) \leq \epsilon_n|u_n'(x_0)| \leq 1$.  We hence obtain $|v_{n,i}(x)+1| \leq 4$ for all $x \in (x_i,b_i)$, and so taking $C' = 4^2$ we get
$$\int_{x_i}^{b_i} (v_{n,i}^2-1)^2\,dx = \int_{x_i}^{b_i} (v_{n,i}-1)^2(v_{n,i}+1)^2\,dx \leq C'\int_{x_i}^{b_i} (v_{n,i}-1)^2,$$
as we wanted.

We may now add \eqref{eq:vbound1} to \eqref{eq:vbound2} and apply \eqref{eq:transition} to obtain
\begin{align*}
    &\phantom{{}={}}\int_{x_0}^{b_i} \epsilon_n^{-1}(v_{n,i}^2-1)^2 + \epsilon_n^3|v_{n,i}''|^2\,dx \\
    &\leq \max(C_0,C_2)\max(1,C')\left[\frac{T}{\epsilon_n}(A^2+m^2T^2) + \frac{\epsilon_n^3}{T^3}(A^2+m^2T^2)\right] \\
 &= C_3(A^2+m^2T^2)\left(\frac{T}{\epsilon_n} + \frac{\epsilon_n^3}{T^3}\right)
\end{align*}
for some constant $C_3 > 0$.  

Since $\frac23y_{j-1}+\frac13y_j < x_0 < \frac13y_{j-1}+\frac23y_j$, we have that
$(b_i-x_i)\epsilon_n \geq \frac{b_i-x_i}{K_n} = y_j-y_{j-1} > y_j - x_0 = T > y_j - (\frac13y_{j-1}+\frac23 y_j) = \frac13(y_j-y_{j-1}) = \frac{b_i-x_i}{3K_n} \geq \frac{b_i-x_i}{3} \cdot \frac{\epsilon_n}{1+\epsilon_n}$.  So we may use this to write
\begin{align*}
 &\phantom{{}={}}\int_{x_0}^{b_i} \epsilon_n^{-1}(v_{n,i}^2-1)^2 + \epsilon_n^3|v_{n,i}''|^2\,dx \\
&\leq C_3((u_n(x_0)-1)^2+|u_n'(x_0)|^2T^2)\left(\frac{T}{\epsilon_n} + \frac{\epsilon_n^3}{T^3}\right) \\
&\leq C_3((u_n(x_0)-1)^2+|u_n'(x_0)|^2(b_i-x_i)^2\epsilon_n^2)\left(\frac{(b_i-x_i)\epsilon_n}{\epsilon_n} + \epsilon_n^3 \cdot \frac{27(1+\epsilon_n)^3}{(b_i-x_i)^3\epsilon_n^3}\right) \\
&\leq C_3((u_n(x_0)-1)^2+|u_n'(x_0)|^2(b_i-x_i)^2\epsilon_n^2)\left(b_i-x_i + \frac{27 \cdot 8}{(b_i-x_i)^3}\right) \\
&\leq C_4((u_n(x_0)-1)^2+|u_n'(x_0)|^2\epsilon_n^2) \leq C_4\eta,
\end{align*}
Where $C_4 > 0$ is a constant with no dependence on $n$, and we have applied our choice of $x_0$.  At last, it follows that
\begin{align*}
&\phantom{{}={}}\int_{x_i}^{b_i} \epsilon_n^{-1}(v_{n,i}^2-1)^2 + \epsilon_n^3|v_{n,i}''|^2\,dx \\
&= \int_{x_i}^{x_0} \epsilon_n^{-1}(v_{n,i}^2-1)^2 + \epsilon_n^3|v_{n,i}''|^2\,dx + \int_{x_0}^{b_i} \epsilon_n^{-1}(v_{n,i}^2-1)^2 + \epsilon_n^3|v_{n,i}''|^2\,dx \\
&= \int_{x_i}^{x_0} \epsilon_n^{-1}(u_n^2-1)^2 + \epsilon_n^3|u_n''|^2\,dx + \int_{x_0}^{b_i} \epsilon_n^{-1}(v_{n,i}^2-1)^2 + \epsilon_n^3|v_{n,i}''|^2\,dx \\
&= \int_{x_i}^{x_0} \epsilon_n^{-1}(u_n^2-1)^2 + \epsilon_n^3|u_n''|^2\,dx + C_4\eta \\
&\leq \int_{x_i}^{b_i} \epsilon_n^{-1}(u_n^2-1)^2 + \epsilon_n^3|u_n''|^2\,dx + C_4\eta. 
\end{align*}

Doing the same thing for the interval $(a_i,x_i)$, we finally obtain
\begin{equation}
    \int_{a_i}^{b_i} \epsilon_n^{-1}(v_{n,i}^2-1)^2 + \epsilon_n^3|v_{n,i}''|^2\,dx \leq \int_{a_i}^{b_i} \epsilon_n^{-1}(u_n^2-1)^2 + \epsilon_n^3|u_n''|^2\,dx + 2C_4\eta. \label{eq:key}
\end{equation}
Thus the key claim has been proven.

\textbf{Step 3}

We may now complete the liminf argument.  We recall the definitions of the family $\mathscr{J}$ and the constant $\alpha$ from $\eqref{eq:famJ}$ and $\eqref{eq:alpha}$ respectively.

Let us first write
$$L+\eta \geq F_{\epsilon_n}(u_n) \geq \sum_{i=1}^J \int_{a_i}^{b_i} \epsilon_n^{-1}(u_n^2-1)^2 + \epsilon_n^3|u_n''|^2\,dx$$
for all large enough $n$.  We now apply the key claim \eqref{eq:key} to every interval $(a_i,b_i)$ to get
$$L+\eta \geq -2C_4J\eta+\sum_{i=1}^J \int_{a_i}^{b_i} \epsilon_n^{-1}(v_{n,i}(x)^2-1)^2 + \epsilon_n^3|v_{n,i}''(x)|^2\,dx.$$
We now apply the change of variables $y = \frac{x-a_i}{b_i-a_i}$.  Defining $w_{n,i}:[0,1] \to \R$ via $w_{n,i}(y) = v_{n,i}\left((b_i-a_i)y+a_i\right)$, we have that $w_{n,i}''(y) = (b_i-a_i)^2v_{n,i}''\left((b_i-a_i)y+a_i\right)$, so that we get
\begin{align*}
L+\eta &\geq -2C_4J\eta+\sum_{i=1}^J (b_i-a_i)\int_0^1 \epsilon_n^{-1}(w_{n,i}(y)^2-1)^2 + \frac{\epsilon_n^3}{(b_i-a_i)^4}|w_{n,i}''(y)|^2\,dy \\
&= -2C_4J\eta+\sum_{i=1}^J \frac{b_i-a_i}{\epsilon_n}\int_0^1 (w_{n,i}(y)^2-1)^2\,dy + \frac{\epsilon_n^3}{(b_i-a_i)^3}\int_0^1|w_{n,i}''(y)|^2\,dy,
\end{align*}
and by applying the weighted AM-GM inequality we may go down again to get
\begin{align*}
  L + \eta &\geq -2C_4J\eta+\sum_{i=1}^J \frac{4}{3^{3/4}}\left(\int_0^1 (w_{n,i}(y)^2-1)^2\,dy\right)^{3/4}\left(\int_0^1|w_{n,i}''(y)|^2\,dy\right)^{1/4}  \\
  &= -2C_4J\eta + \sum_{i=1}^J \frac{4}{3^{3/4}}\Phi(w_{n,i}).
\end{align*}
Lastly, as $w_{n,i} \in \mathscr{J}$ for all $i$, we have that $\frac{4}{3^{3/4}}\Phi(w_{n,i}) \geq \frac{4}{3^{3/4}}\inf_{w \in \mathscr{J}} \Phi(w) = 2\alpha$, thus
\begin{align*}
  L + \eta &\geq -2C_4J\eta + \sum_{i=1}^J 2\alpha \\
  &= -2C_4J\eta + 2J\alpha = -2C_4J\eta + \alpha\essVar_\Omega u.
\end{align*}
Hence $L+\eta \geq -2C_4J\eta + \alpha\essVar_\Omega u$.  As $\eta$ was arbitrary, we conclude that $L \geq \alpha\essVar_\Omega u$, as needed.  This proves the liminf inequality for $u \in BPV(\Omega;\pm 1)$.

\noindent\textbf{Step 4}

In this step, we briefly resolve the case $u \in L^2(\Omega;\pm 1) \setminus BPV(\Omega;\pm 1)$.

Showing that $F_{\epsilon_n}(u_n) \to +\infty$ in this case (which is sufficient to conclude the Gamma limit) requires modification of the argument in Steps 2 and 3.  We sketch the proof:  Find $2N$ Lebesgue points $a < x_1 < y_1 < \cdots < x_N < y_N < b$ for which $u(x_i) = -1$ and $u(y_i) = 1$.  For each $i$ we may find an interval $(x_i,x_i+\delta)$ for which $u = -1$ in most of $(x_i,\delta)$ and an interval $(y_i-\delta,y_i)$ for which $u = 1$ in most of $(y_i-\delta,y_i)$.  We can then apply the pigeonhole argument in these intervals by using the fact that $\|u_n+1\|_{L^p((x_i,x_i+\delta) \cap \{u = -1\}} \to 0$ and $\|u_n-1\|_{L^p((y_i-\delta,y_i) \cap \{u = 1\}} \to 0$.  In the end, we obtain a lower bound for $F_{\epsilon_n}(u_n)$ that tends to $+\infty$ when we send $N \to +\infty$.

\noindent\textbf{Step 5}

The remaining thing to prove is the limsup inequality for $u \in BPV(\Omega;\pm 1)$.  This entails finding a sequence $u_n \in W^{2,p}(\Omega)$ for which $u_n \to u$ in $L^p(\Omega)$ and $\limsup_{n \to \infty} F_{\epsilon_n}(u_n) \leq \alpha \Var_\Omega u$.  

First, recall the definition of $\Phi$ as in \eqref{eq:Phi} and the family $\mathscr{J}$ as in \eqref{eq:famJ}.  Define the subfamily
$$\mathscr{J}_n := \left\{h \in \mathscr{J} : \frac{\int_0^1 |h''|^2\,dx}{\int_0^1 (h^2-1)^2\,dx} \leq \frac{1}{\epsilon_n}\right\}.$$
Since $\epsilon_n \to 0^+$, it is clear that $\bigcup_{n=1}^\infty \mathscr{J}_n = \mathscr{J}$.
Thus if we define the sets
$$S_n := \left\{\Phi(h) : h \in \mathscr{J}_n\right\},$$
$$S := \left\{\Phi(h) : h \in \mathscr{J}\right\},$$
then we have $\bigcup_{n=1}^\infty S_n = S$, and it is not hard to show that $\lim_{n \to \infty} \inf S_n = \inf S$.
We note also that $\alpha$, as defined in \eqref{eq:alpha}, may be written as $\alpha = \frac{2}{3^{3/4}} \inf S$.

For all $n$, find $h_n \in \mathscr{J}_n$ for which
\begin{equation}
  \inf S_n \leq \Phi(h_n) \leq \frac1n + \inf S_n.  \label{eq:hbound}
\end{equation}
As in Step 2, suppose $u$ ``jumps" at the points $x_1,x_2,\cdots,x_J \in \Omega$, and find pairwise disjoint intervals $(a_i,b_i) \subseteq \Omega$ with $x_i \in (a_i,b_i)$.  Consider some $x_i$, and without loss of generality assume that $u$ jumps from $-1$ to $1$ at $x_i$.  Let $\zeta_n \to 0$ be a sequence of positive reals that we shall choose later.  We now define $u_n$ over $(a_i,b_i)$ as
$$u_n(x) := \begin{cases}-1, & a_i < x < x_i - \zeta_n/2 \\ h_n\left(\frac{x-x_i+\zeta_n/2}{\zeta_n}\right), & x_i-\zeta_n/2 \leq x \leq x_i+\zeta_n/2 \\ 1, & x_i+\zeta_n/2 < x < b_i\end{cases}.$$
We define $u_n$ similarly over all other intervals $(a_j,b_j)$, and in $\Omega \setminus \bigcup_{j=1}^J (a_j,b_j)$ we let $u_n$ agree with $u$.  Since $h_n(0_+) = -1$, $h_n(1_-) = 1$, and $h_n'(0_+) = h_n'(1_-) = 0$, we have that $u_n \in W^{2,p}(\Omega)$.  Moreover, $u_n \to u$ almost everywhere as a consequence of $\zeta_n \to 0^+$, and $\{|u_n|\}_n$ is uniformly integrable, so by Vitali we have $u_n \to u$ in $L^p(\Omega)$.  It remains to verify the limsup inequality.  We may compute
\begin{align*}
F_{\epsilon_n}(u_n) &= \sum_{i=1}^J \int_{a_i}^{b_i} \epsilon_n^{-1}(u_n^2-1)^2+\epsilon_n^3|u_n''|^2\,dx \\
&=\sum_{i=1}^J \int_{a_i}^{b_i} \epsilon_n^{-1}\left(h_n\left(\frac{x-x_i+\zeta_n/2}{\zeta_n}\right)^2-1\right)^2+\frac{\epsilon_n^3}{\zeta_n^4}\left|h_n''\left(\frac{x-x_i+\zeta_n/2}{\zeta_n}\right)\right|^2\,dx,
\end{align*}
and after changing variables we get
\begin{align}
F_{\epsilon_n}(u_n)&=\sum_{i=1}^J \zeta_n\int_0^1 \epsilon_n^{-1}\left(h_n\left(y\right)^2-1\right)^2+\frac{\epsilon_n^3}{\zeta_n^4}\left|h_n''\left(y\right)\right|^2\,dy \nonumber \\
&=\sum_{i=1}^J \frac{\zeta_n}{\epsilon_n}\int_0^1 \left(h_n\left(y\right)^2-1\right)^2\,dy+\frac{\epsilon_n^3}{\zeta_n^3}\int_0^1\left|h_n''\left(y\right)\right|^2\,dy. \label{eq:limsup1}
\end{align}
We now choose $\zeta_n$ so that we obtain the equality case in the AM-GM inequality.  Specifically, we choose
$$\zeta_n = \epsilon_n\left(\frac{3\dint_0^1\left|h_n''\left(y\right)\right|^2\,dy}{\dint_0^1 \left(h_n\left(y\right)^2-1\right)^2\,dy}\right)^{1/4}.$$
Showing that this choice is valid for all sufficiently large $n$ amounts to proving that $\zeta_n \to 0^+$.  Fortunately, since $h_n \in \mathscr{J}_n$, we have that
$$\zeta_n \leq \epsilon_n\left(\frac{3}{\epsilon_n}\right)^{1/4} = 3^{1/4}\epsilon_n^{3/4} \to 0^+.$$
With the choice of $\zeta_n$ justified, we may now continue the computation in \eqref{eq:limsup1} by using the choice of $\zeta_n$ and \eqref{eq:hbound} to write
$$F_{\epsilon_n}(u_n) = \sum_{i=1}^J \frac{4}{3^{3/4}}\left(\int_0^1 (h_n^2-1)^2\,dy\right)^{3/4}\left(\int_0^1|h_n''|^2\,dy\right)^{1/4} \leq \frac{4J}{3^{3/4}}\left(\frac1n+\inf S_n\right).$$
Taking the limsup, we obtain
$$\limsup_{n \to \infty} F_{\epsilon_n}(u_n) \leq \limsup_{n \to \infty} \frac{4J}{3^{3/4}}\left(\frac1n+\inf S_n\right) = \frac{2 \Var_\Omega u}{3^{3/4}}\left(\inf S\right) = \alpha\Var_\Omega u.$$
This completes the proof.
\end{proof}

\section{Boundary Conditions}

The Gamma limit will change upon restricting to the boundary conditions $u(a_+) = a_\epsilon$ and $u(b_-) = b_\epsilon$ for $a_\epsilon \to -1$ and $b_\epsilon \to 1$ as $\epsilon \to 0^+$.  A portion of the work needed to account for the boundary conditions has already been done in the proof of Theorem \ref{gamma2}, but there is still much to be done.  Intuitively, every jump in the interior of $(a,b)$ induces a factor of $\alpha$, whereas a jump at the boundary induces a factor of $\beta(t)$ depending on the height of the jump.  Recall that $\beta(t)$ is defined as in the statement of Theorem \ref{boundary2}, and we will again reference the family $\mathscr{J}'(t)$ and the functional $\Phi$.

We first introduce a new family of functions $\mathscr{J}'(t)$, defined as
$$\mathscr{J}'_\infty(t) := \left\{u \in W^{2,\infty}(-\infty,0) : u(0) = t \text{ and } u(x) = -1 \ \forall x \leq -L \text{ for some }L > 0\right\}$$
for a parameter $t \in \R$.  We associate with each $u \in \mathscr{J}'_\infty(t)$ the constant $L_u$, where $-L_u$ is the first time that $u$ reaches $-1$ and remains at this value indefinitely.  That is,
$$L_u := \inf\left\{L > 0 : u(x) = -1\ \forall x \leq -L\right\}.$$
We also define a new functional $\Psi:W^{2,\infty}(-\infty,0) \to \R$ via
$$\Psi(u) := \int_{-\infty}^0 (u^2-1)^2 + |u''|^2\,dx.$$
The relevance of these constructs is as follows.

\begin{lemma}
$$\beta(t) = \frac{4}{3^{3/4}}\inf_{u \in \mathscr{J}'(t)}\Phi(u) = \inf_{v \in \mathscr{J}'_\infty(t)}\Psi(v).$$
\end{lemma}

\begin{proof}
We first note that if $t=-1$ then clearly $\inf_{u \in \mathscr{J}'(t)} \Phi(u) = 0$ and $\inf_{u \in \mathscr{J}'_\infty(t)} \Psi(u) = 0$, so we may assume that $t \neq -1$.  Now consider $u \in \mathscr{J}'(t)$.  Then take $L > 0$ depending only on $u$ that we shall choose later, and define $v \in \mathscr{J}'_\infty(t)$ via 
$$v(x) = \begin{cases}-1, & x \leq -L \\ u(x/L + 1), & -L < x < 0\end{cases}.$$  
Then
\begin{align}
    \Psi(v) &= \int_{-L}^0 (v(x)^2-1)^2 + |v''(x)|^2\,dx \nonumber\\ &= \int_{-L}^0 (u(x/L+1)^2-1)^2 + \frac{1}{L^4}|u''(x/L+1)|^2\,dx \nonumber\\&= \int_0^1 L(u(y)^2-1)^2 + \frac{1}{L^3}|u''(y)|^2\,dy. \label{eq:famchange}
\end{align}
By examination of the equality case of AM-GM, this is precisely $\frac{4}{3^{3/4}}\Phi(u)$ for a proper choice of $L$.  Specifically, one may take $L = \left(\frac{3\int_0^1 |u''|^2\,dx}{\int_0^1 (u^2-1)^2\,dx}\right)^{1/4}$, which is well-defined from the assumption that $t \neq -1$.  This gives $\frac{4}{3^{3/4}}\Phi(u) = \Psi(v) \geq \inf_{v \in \mathscr{J}_\infty'(t)} \Psi(v)$, and taking the infimum gives $\inf_{u \in \mathscr{J}'(t)} \Phi(u) \geq \inf_{v \in \mathscr{J}_\infty'(t)} \Psi(v)$.

On the other hand, if $v \in \mathscr{J}'_\infty(t)$, then from taking $u(x) := v\left(L_v(x-1)\right)$ we get from using the same computations done in \eqref{eq:famchange} and applying AM-GM that
\begin{align*}
    \Psi(v) &= \int_0^1 \frac{L_v}{3}(u(y)^2-1)^2 + \frac{1}{L_v^3}|u''(y)|^2\,dy \\
    &\geq \frac{4}{3^{3/4}}\left(\int_0^1 (u(y)^2-1)^2\,dy\right)^{3/4}\left(\int_0^1|u''(y)|^2\,dy\right)^{1/4} \\
    &= \frac{4}{3^{3/4}}\Phi(u),
\end{align*}
so that $\Psi(v) \geq \frac{4}{3^{3/4}} \inf_{u \in \mathscr{J}'(t)} \Phi(u)$.  Thus $\inf_{v \in \mathscr{J}'_\infty(t)} \Psi(v) \geq \frac{4}{3^{3/4}} \inf_{u \in \mathscr{J}'(t)} \Phi(u)$, finishing the proof.
\end{proof}


We will now prove several crucial results concerning the family $\mathscr{J}'_\infty(t)$ and the functional $\Psi$.  The first is a compactness result.

\begin{lemma}\label{psicompact}
Let $\Omega \subseteq \R$ be a non-empty bounded open set and $u_n \in w^{2,\infty}(\Omega)$ such that 
$$M := \sup_{n \in \N} \int_\Omega (u_n^2-1)^2 + |u_n''|^2\,dx < \infty.$$
There exists a subsequence $u_{n_k}$ and $u \in W^{2,2}(\Omega)$ such that
\begin{enumerate}
    \item $u_{n_k} \wc u$ in $W^{2,2}(\Omega)$,
    \item $u_{n_k}' \to u'$ almost everywhere and in $L^2(\Omega)$, and
    \item $u_{n_k} \to u$ almost everywhere and in $L^2(\Omega)$
\end{enumerate}
\end{lemma}

\begin{proof}
Find $A > 0$ large enough so that $z^2 \leq (z^2-1)^2$ for all $|z| \geq A$.  Then
$$\int_\Omega u_n^2\,dx \leq \int_{\{|u_n| < A\}} u_n^2\,dx + \int_{\{|u_n| \geq A\}} u_n^2\,dx \leq A^2\leb^1(\Omega)+ A\int_\Omega (u_n^2-1)^2 \leq A^2\leb^1(\Omega) + AM,$$
so $\{u_n\}_n$ is uniformly bounded in $L^2(\Omega)$.  Since $\{u_n''\}_n$ is uniformly bounded in $L^2(\Omega)$ as well, we may apply Theorem \ref{inter2} with some $l < \leb^1(\Omega)$ to deduce that $\{u_n'\}_n$ is uniformly bounded in $L^2(\Omega)$.  Hence $\{u_n\}_n$ is uniformly bounded in $W^{2,2}(\Omega)$.

By Theorem \ref{weakcompactness}, there exists a subsequence $u_{n_k}$ and $u \in W^{2,2}(0,T)$ such that $u_{n_k} \wc u$ in $W^{2,2}(0,T)$.  Moreover, by \ref{weakstronger}, we have that $u_{n_k} \to u$ strongly in $L^2(0,T)$ and $u_{n_k}' \to u'$ strongly in $L^2(0,T)$.  By extracting a subsequence, we may also get $u_{n_k} \to u$ and $u_{n_k}' \to u'$ almost everywhere.  
\end{proof}

Next, we show that we may ``force" a bound on $L_u$ for functions $u \in \mathscr{J}'_\infty(t)$.

\begin{lemma}\label{boundL}
Let $M,\eta > 0$.  Then there exists a constant $L_{M,\eta} > 0$ depending only on $M$ and $\eta$ such that for every $t \in \R$ and $u \in \mathscr{J}'_\infty(t)$ with $\Psi(u) \leq M < \infty$, there exists $v \in \mathscr{J}'_\infty(t)$ such that $\Psi(v) \leq \Psi(u) + O(\eta)$ and $v(x) = -1$ for all $x \leq -L_{M,\eta}$.  That is, $L_v \leq L_{M,\eta}$. 
\end{lemma}

\begin{proof}
We begin by constructing a function $v_1 \in \mathscr{J}'_\infty(t)$ such that $\Psi(v_1) \leq \Psi(u) + O(\eta)$ and $\leb^1(v_1^{-1}((0,2))) \leq C$ for a constant $C$ depending only on $M$ and $\eta$. 

Let $K = \frac{M}{\eta}$.  We will take $C = 2K+2$.  Let $E := u^{-1}((0,2))$.  If we have that $\leb^1(E) \leq 2K+2$, then we simply take $v_1 = u$.  Otherwise, we may let $E_1 := (-\infty,y) \cap E$, $F := [y,z] \cap E$, and $E_2 := (z,0) \cap E$, where $y,z$ are chosen such that $\leb^1(E_1) = \leb^1(E_2) = K$ and $\leb^1(F) = \leb^1(E) - 2K > 2$.

Since $u > 0$ over $E_1$, we may write
$$M \geq \int_{E_1} (u^2-1)^2 + |u''|^2\,dx \geq \int_{E_1} (u-1)^2 + |u''|^2\,dx,$$
so there exists $x_1 \in E_1$ such that 
$(u(x_1)-1)^2 + u''(x_1)^2 \leq \frac{M}{\leb^1(E_1)} = \frac{M}{K} = \eta$.
It follows by Lemma \ref{glue} that there exists $\tilde{u}_1 \in C^\infty([x_1,x_1+1])$ such that $\tilde{u}_1(x_1) = u(x_1)$, $\tilde{u}_1'(x_1) = u'(x_1)$, $\tilde{u}_1(x_1+1) = 1$, and $\tilde{u}_1'(x_1+1) = 0$, with 
\begin{align*}
  \int_{x_1}^{x_1+1} (\tilde{u}_1^2-1)^2 + |\tilde{u}''|^2\,dx &\leq \int_{x_1}^{x_1+1} 9(\tilde{u}_1-1)^2\,dx + \int_{x_1}^{x_1+1} |\tilde{u}''|^2\,dx \\
  &\leq 9C'(u(x_1)-1)^2 + C'|u'(x_1)|^2 \\
  &\leq 9C'\eta = O(\eta)
\end{align*}
for some universal constant $C' > 0$.  We construct $x_2 \in E_2$ and $\tilde{u}_2 \in C^\infty([x_2-1,x_2])$ in a similar manner so that $\int_{x_2-1}^{x_2} (\tilde{u}_2-1)^2 + |\tilde{u}_2''|^2\,dx = O(\eta)$.

Now define $\tilde{u}$ as
$$\tilde{u}(x) := \begin{cases}u(x), & -\infty < x < x_1 \\ \tilde{u}_1(x), & x_1 \leq x \leq x_1+1 \\ 1, & x_1+1 < x < x_2-1 \\ \tilde{u}_2(x), & x_2-1 \leq x \leq x_2 \\ u(x), & x_2 < x < 0\end{cases}.$$
Note that this is well-defined in the sense that $x_1+1 < x_2-1$ because $x_2-x_1 > z-y \geq \leb^1(F) \geq 2$.  Moreover our choices for $\tilde{u}_1$ and $\tilde{u}_2$ ensure that $u \in W^{2,\infty}(-\infty,0)$, and
\begin{align*}
    \Psi(\tilde{u}) &= \int_{(-\infty,x_1) \cup (x_2,0)} (\tilde{u}^2-1)^2 + |\tilde{u}''|^2\,dx + \int_{[x_1,x_1+1] \cup [x_2-1,x_2]} (\tilde{u}^2-1)^2 + |\tilde{u}''|^2\,dx \\
    &\leq \int_{(-\infty,x_1) \cup (x_2,0)} (u^2-1)^2 + |\tilde{u}''|^2\,dx + \int_{[x_1,x_1+1] \cup [x_2-1,x_2]} (u^2-1)^2 + |u''|^2\,dx + O(\eta) \\
    &\leq \Psi(u) + O(\eta).
\end{align*}
We may now take $v_1$ to be
$$v_1(x) := \begin{cases}\tilde{u}(x-x_2+x_1+2), & -\infty < x < x_2-1 \\ \tilde{u}(x), & x_2-1 < x < 0\end{cases}.$$
In essence, we have ``deleted" an interval in which $\tilde{u} = 1$.  We still have $v \in W^{2,\infty}(-\infty,0)$ and $v(0_+) = u(0_+) = t$, with $\lim_{x \to -\infty} v(x) = -1$, so that $v \in \mathscr{J}'_\infty(t)$.  Furthermore, $\Psi(v) = \Psi(\tilde{u}) \leq \Psi(u) + O(\eta)$.  Lastly, we see that
\begin{align*}
  \leb^1(\{x < 0 : 0 < v_1(x) < 2\}) &\leq \leb^1(\{x \in (-\infty,x_1+1) \cup (x_2-1,0) : 0 < \tilde{u}(x) < 2\})  \\
  & \leq \leb^1(\{x \in E_1 \cup E_2 : 0 < u(x) < 2\})+2 \\
  & = 2K+2 = C,
\end{align*}
which was our goal.

Now we can construct $v$.  Observe that
$$M \geq \int_{\{\min(|v-1|,|v+1|) > 1\}} (v_1^2-1)^2+|v_1''|^2\,dx \geq \leb^1(\{x < 0 : \min(|v(x)-1|,|v(x)+1|) > 1\}).$$
Thus, if we take $L_{M,\eta} = C + M + K+1$, then the set $$G := \{-L_{M,\eta}+1 < x < 0 : -2 < v(x) < 0\}$$ satisfies $\leb^1(G) \geq K$.  Since $\int_G (v_1+1)^2+|v''|^2\,dx \leq \int_G (v_1^2-1)^2+|v''|^2\,dx \leq M$, there exists $x_3 \in G$ such that 
$$(v_1(x_3)+1)^2 + |v_1''(x_3)|^2 \leq \frac{M}{\leb^1(G)} \leq M/K=\eta.$$
Hence, as we did before to $u$, we may use Lemma \ref{glue} to modify $v_1$ in the interval $(x_3-1,x_3)$ and hence obtain a function $v$ for which $v(x) = -1$ for all $x < x_3-1$, $v(x) = v_1(x)$ for $x_3 < x < 0$, and $\int_{x_3-1}^{x_3} (v_1^2-1)^2+|v_1''|^2\,dx \leq O(\eta)$.  

This function $v$ satisfies $\Psi(v) \leq \Psi(v_1) + O(\eta) \leq \Psi(u) + O(\eta)$ and, since $-L_{M,\eta}+1 < x_3$, we have that $v(x) = -1$ for all $x < -L_{M,\eta}$, where $L_{M,\eta}$ depends only on $M$ and $\eta$, as desired.
\end{proof}

Now we arrive at our first major result.

\begin{lemma}\label{betacts}
Let $t_0 \in \R$. Then we have that
$$\lim_{t \to t_0} \beta(t) = \beta(t_0),$$
where $\beta(t)$ is defined as in \eqref{eq:beta}.  That is, $\beta$ is continuous.
\end{lemma}

\begin{proof} $ $\newline

\noindent\textbf{Step 1}

We claim that $\beta(t)$ is decreasing over $t \leq -1$ and increasing over $t \geq -1$.  

First note that $\beta(t) \geq 0$ for all $t$ and $\beta(-1) = 0$.  Now suppose $s,t \in \R$ satisfy either $-1 < s \leq t$ or $t \leq s < -1$.  We show that $\beta(s) \leq \beta(t)$ by proving that for all $v \in \mathscr{J}'(t)$ we may find $u \in \mathscr{J}'(s)$ for which $\Phi(u) \leq \Phi(v)$.

Indeed, since $v(1_-) = t$ and $v(0_+) = -1$, there must exist $T \in (0,1)$ for which $v(T) = s$.  Now take $u(x) := v(Tx)$.  Evidently, $u \in \mathscr{J}'(s)$.  Moreover we have
$$\int_0^1 (u^2-1)^2\,dx = \int_0^T T^{-1}(v^2-1)^2\,dx \leq \int_0^1 T^{-1}(v^2-1)^2\,dx$$
and
$$\int_0^1 |u''|^2\,dx = \int_0^T T^3|v''|^2\,dx \leq \int_0^1 T^3|v''|^2\,dx.$$
Thus
\begin{align*}
  \Phi(u) &= \left(\int_0^1 (u^2-1)^2\,dx\right)^{3/4}\left(\int_0^1 |u''|^2\,dx\right)^{1/4} \\
  &\leq \left(\int_0^1 T^{-1}(v^2-1)^2\,dx\right)^{3/4}\left(\int_0^1 T^3|v''|^2\,dx\right)^{1/4}\\
  &= \left(\int_0^1 (v^2-1)^2\,dx\right)^{3/4}\left(\int_0^1 |v''|^2\,dx\right)^{1/4}\\
  &= \Phi(v),  
\end{align*}
as needed.

\noindent\textbf{Step 2}

Take $t_0 \in \R$.  We show that $\limsup_{t \to t_0} \beta(t) \leq \beta(t_0)$.  It is sufficient to prove that for any $v \in \mathscr{J}'(t_0)$, we may pick $u_t \in \mathscr{J}'(t)$ for each $t \in \R$ such that $\lim_{t \to t_0} \Phi(u_t) = \Phi(v)$.  This is because if we fix $\eta > 0$, then we may take $v \in \mathscr{J}'(t_0)$ such that $\beta(t_0) \leq \frac{4}{3^{3/4}}\Phi(v) \leq \beta(t_0) + \eta$, take $u_t$ for each $t \in \R$ as above, and then choose $\delta > 0$ so small that $|\Phi(v) - \Phi(u_t)| < \eta$ for all $t$ with $|t-t_0| < \delta$, so that 
$$\beta(t) \leq \frac{4}{3^{3/4}}\Phi(u_t) \leq \frac{4}{3^{3/4}}(\Phi(v)+\eta) \leq \beta(t_0) + \eta+\frac{4}{3^{3/4}}\eta,$$
which is enough.

For $v \in \mathscr{J}'(t_0)$, we take $u_t \in \mathscr{J}'(t)$ to be $u_t := \frac{1+t}{1+t_0}(v+1)-1$.  Then
$$\int_0^1 (u_t^2-1)^2 \,dx = \left(\frac{1+t}{1+t_0}\right)^4\int_0^1 (v+1)^2\left(v+1-2 \cdot \frac{1+t_0}{1+t}\right)^2\,dx.$$
Taking the limit, we obtain
$$\lim_{t \to t_0} \int_0^1 (u_t^2-1)^2\,dx = \lim_{t \to t_0}\int_0^1 (v+1)^2\left(v+1-2 \cdot \frac{1+t_0}{1+t}\right)^2\,dx = \int_0^1 (v^2-1)^2\,dx$$
because $(v+1)^2\left(v+1-2 \cdot \frac{1+t_0}{1+t}\right)^2$ converges to $(v^2-1)^2$ pointwise as $t \to t_0$ and, since $v$ is bounded, we have that $(v+1)^2\left(v+1-2 \cdot \frac{1+t_0}{1+t}\right)^2$ is uniformly bounded for all $t$ sufficiently near $t$, which is enough to apply dominated convergence.

We also have
$$\lim_{t \to t_0} \int_0^1 |u_t''|^2\,dx = \lim_{t \to t_0} \left(\frac{1+t}{1+t_0}\right)^2\int_0^1 |v''|^2\,dx = \int_0^1 |v''|^2\,dx.$$
Altogether, we see that
\begin{align*}
  \lim_{t \to t_0} \Phi(u_t) &= \left(\lim_{t \to t_0} \int_0^1 (u_t^2-1)^2\,dx\right)^{3/4}\left(\lim_{t \to t_0} \int_0^1|u_t''|^2\,dx\right)^{1/4}\\ &= \left(\int_0^1(v^2-1)^2\,dx\right)^{3/4}\left(\int_0^1|v''|^2\,dx\right)^{1/4}\\ &= \Phi(v).  
\end{align*}
This completes the proof that $\beta$ is upper semi-continuous.  

\noindent\textbf{Step 3}

We now define yet another family $\overline{\mathscr{J}'_\infty(t)}$ via
$$\overline{\mathscr{J}'_\infty(t)} := \left\{u \in W^{2,\infty}(-\infty,0) : u(0_-) = t \text{ and } \lim_{x \to -\infty} u(x) = -1\right\}.$$
We claim that 
$$\beta(t) = \inf_{u \in \mathscr{J}'_\infty(t)} \Psi(u) = \inf_{u \in \overline{\mathscr{J}'_\infty(t)}} \Psi(u).$$
Since $\mathscr{J}'_\infty(t) \subseteq \overline{\mathscr{J}'_\infty(t)}$, it is clear that $\inf_{u \in \mathscr{J}'_\infty(t)} \Psi(u) \geq \inf_{u \in \overline{\mathscr{J}'_\infty(t)}} \Psi(u)$.  To show that $\inf_{u \in \mathscr{J}'_\infty(t)} \Psi(u) \leq \inf_{u \in \overline{\mathscr{J}'_\infty(t)}} \Psi(u)$, we fix $\epsilon > 0$ and $u \in \overline{\mathscr{J}'_\infty(t)}$.  It is then sufficient to find $v \in \overline{\mathscr{J}'_\infty(t)}$ for which $\Psi(v) \leq \Psi(u) + O(\epsilon)$.

Find $M$ so large that $\int_{-\infty}^{-M} |u''|^2\,dx < \epsilon^2$ and $|u(x)+1| < \epsilon$ for all $x \leq -M$.  Then
$$\int_{-M-1}^{-M} (u+1)^2\,dx \leq \epsilon^2,$$
and so by Theorem \ref{inter2} applied to $u+1$ with $l=1$, we see that
$$\left(\int_{-M-1}^{-M} |u'|^2\,dx\right)^{1/2} \leq c\left(\int_{-M-1}^{-M} (u+1)^2\,dx\right)^{1/2}+c\left(\int_{-M-1}^{-M} |u''|^2\,dx\right)^{1/2} \leq 2c\epsilon$$
for a universal constant $c > 0$, so that $\int_{-M-1}^{-M} |u'|^2\,dx \leq 4c^2\epsilon^2$.  It follows that there exists $T \in (M,M+1)$ such that $|u'(-T)| \leq 2c\epsilon$.  Since $|u(-T)+1| < \epsilon$, we may apply Lemma \ref{glue} to find $f \in C^\infty([-T-1,-T])$ such that 
\begin{enumerate}[label=(\roman*)]
    \item $f(-T-1) = f'(-T-1) = 0$, $f(-T) = u(-T)$, $f'(-T) = u'(-T)$, 
    \item for a universal constant $C_0$ we have $$\int_{-T-1}^{-T} |f+1|^2\,dx \leq C_0(|u(-T)+1|^2 + |u'(-T)|^2) = C_0(\epsilon^2 + 4c^2\epsilon^2),$$ 
    \item for a universal constant $C_2$ we have $$\int_{-T-1}^{-T} |f''|^2\,dx \leq C_2(|u(-T)+1|^2 + |u'(-T)|^2) = C_2(\epsilon^2 + 4c^2\epsilon^2),$$ and
    \item $\sup_{[-T-1,-T]} |f| \leq |u(-T)+1| + \frac12|u'(-T)| = (1+2c)\epsilon$.
\end{enumerate}

Naturally we now take
$$v(x) := \begin{cases}0, & x < -T-1 \\ f(x), & -T-1 \leq x \leq -T \\ u(x), & -T < x < 0\end{cases}.$$
By item (i), we see that $v \in W^{2,\infty}(-\infty,0)$, and so $v \in \mathscr{J}'_\infty(t)$.

For $\epsilon$ sufficiently small, we have $(1+2c)\epsilon < 1$, which ensures that $|f(x)-1| \leq 2$ for all $x \in [-T-1,-T]$ by item (iv).  This bound and item (ii) gives
\begin{align}
  \int_{-\infty}^0 (v^2-1)^2\,dx &\leq \int_{-\infty}^0 (u^2-1)^2\,dx + \int_{-T-1}^{-T} (f^2-1)^2\,dx \nonumber \\
  &\leq \int_{-\infty}^0 (u^2-1)^2\,dx + 4\int_{-T-1}^{-T} (f+1)^2\,dx \nonumber \\
  &\leq \int_{-\infty}^0 (u^2-1)^2\,dx + 4C_0(\epsilon^2 + 4c^2\epsilon^2)\,dx, \label{eq:vpsibound1}
\end{align}
and item (iii) gives
\begin{align}
    \int_{-\infty}^0 |v''|^2\,dx &\leq \int_{-\infty}^0 |u''|^2\,dx + \int_{-T-1}^{-T} |f''|^2\,dx \nonumber \\
    &\leq \int_{-\infty}^0 |u''|^2\,dx + C_2(\epsilon^2 + 4c^2\epsilon^2). \label{eq:vpsibound2}
\end{align}
Adding \eqref{eq:vpsibound1} and \eqref{eq:vpsibound2} gives
$$\Psi(v) \leq \Psi(u) + (4C_0+C_2)(\epsilon^2 + 4c^2\epsilon^2) = O(\epsilon),$$
which is enough.  This proves the claim.

\noindent\textbf{Step 4}

In this step, we show that if $u \in W^{2,\infty}(-\infty,0)$ such that $\Psi(u) < \infty$, then $\lim_{x \to -\infty} u(x) \in \{-1,1\}$.

It is easy to see that this is true provided that the limit $\lim_{x \to -\infty} u(x)$ exists.  To show that the limit exists, let $l_1 := \liminf_{x \to -\infty} u(x)$ and $l_2 := \limsup_{x \to -\infty} u(x)$, and suppose for contradiction that $l_1 < l_2$.  Then we may find $k_1,k_2$ with $l_1 < k_1 < k_2 < l_2$ such that $k_1,k_2 \not\in \{-1,1\}$.  The continuity of $u$ ensures that $u^{-1}(k_1)$ and $u^{-1}(k_2)$ are unbounded.

Now consider the quantity
$$\gamma := \inf\left\{\int_0^T (v^2-1)^2 + |v''|^2\,dx : v \in W^{2,\infty}(0,T),  T > 1, v(0_+) = k_1, v(T_-) = k_2\right\}.$$
We claim that $\gamma = 0$.  Indeed, we may find sequences $x_n,y_n$ such that $u(-x_n) = k_1$, $u(-y_n) = k_2$, and $x_n + 1 < y_n < x_{n+1}$ for all $n$.  But now
$$\infty > \Psi(u) \geq \sum_{n=1}^\infty \int_{-y_n}^{-x_n} (u^2-1)^2 + |u''|^2\,dx \geq \sum_{n=1}^\infty \gamma,$$
which can only be possible if $\gamma = 0$.

Now find sequences $T_n > 1$ and $v_n \in W^{2,\infty}(0,T_n)$ for which $v_n(0_+) = k_1$, $v_n({T_n}_-) = k_2$, and $\int_0^{T_n} (v_n^2-1)^2 + |v_n''|^2\,dx \to 0$.

Since $T_n > 1$, we may obtain the uniform bound $\int_0^1 (v_n^2-1)^2 + |v_n''|^2\,dx \leq M$ for some $M > 0$.  Also,
$$0 = \lim_{n \to \infty} \int_0^1 (v_n^2-1)^2 + |v_n''|^2\,dx \geq \lim_{n \to \infty} \int_0^1 |v_n''|^2\,dx,$$
so we have that $v_n'' \to 0$ in $L^2(0,1)$.  Moreover, by Lemma \ref{psicompact}, there exists a subsequence $v_{n_k}$ and $v \in W^{2,2}(0,1)$ such that $v_{n_k}'' \wc v''$ in $L^2(0,1)$, $v_{n_k}'' \to 0$ in $L^2(0,1)$, and $v_{n_k} \to v$ pointwise almost everywhere.  Now, notice that both $v_{n_k}'' \wc 0$ and $v_{n_k}'' \wc v''$ in $L^2(0,1)$, so $v'' = 0$ almost everywhere.  It follows that $v$ is affine, so that $v(x) = mx+b$ for some $m,b \in \R$.

Now, Fatou's Lemma gives that
$$0 = \lim_{n \to \infty} \int_0^1 (v_n^2-1)^2 + |v_n''|^2\,dx \geq \lim_{n \to \infty} \int_0^1 (v_n^2-1)^2\,dx \geq \int_0^1 (v^2-1)^2\,dx,$$
which implies that $m = 0$ and $b \in \{-1,1\}$.  

However, we claim that $v$ must satisfy $v(0_+) = k_1$, which would imply that $k_1 \in \{-1,1\}$, resulting in a contradiction.  To see this, use the fact that $v_n(0_+) = k_1$ for all $n$, the Fundamental Theorem of Calculus, and H\"older's inequality to write
$$|v_n(x) - k_1| \leq \int_0^x |v_n'(t)|\,dt \leq \left(\int_0^x |v_n'(t)|^2\,dt\right)^{1/2}\sqrt{x} \leq \left(\int_0^1 |v_n'(t)|^2\,dt\right)^{1/2}\sqrt{x}.$$
Then since $v_n'$ converges in $L^p(0,1)$ we must have that $\int_0^1 |v_n'(t)|^2\,dt$ is bounded, so $|v_n(x) - k_1| \leq M_1\sqrt{x}$ for a constant $M_1 > 0$.  Sending $n \to +\infty$ we then have $|v(x) - k_1| \leq M_1\sqrt{x}$ for almost every $x \in (0,1)$, so we may send $x \to 0^+$ along an appropriate sequence to deduce that $|v(0_+) - k_1| = 0$, as needed.

\noindent\textbf{Step 5}

We are now ready to show that $\liminf_{t \to t_0} \beta(t) \geq \beta(t_0)$.  

Letting $t_n \to t_0$ be arbitrary, we just need to show that $\liminf_{n \to \infty} \beta(t_n) \geq \beta(t_0)$.  

By the monotone properties that we have proven about $\beta$ in Step 1, we see that $\{\beta(t_n) : n \in \N\}$ is bounded by a constant $M$.  Specifically, we may take $$M = \max(\beta(\sup_{n \in \N} t_n), \beta(\inf_{n \in \N} t_n)).$$

Next, by extraction of a subsequence, let us assume that there exists the limit $L := \lim_{n \to \infty} \beta(t_n)$.

Now fix $\eta > 0$.  Select $\tilde{u}_n \in \mathscr{J}'_\infty(t_n)$ such that $\beta(t_n) \leq \Psi(\tilde{u}_n) \leq \beta(t_n) + \frac1n$.  Since $\Psi(\tilde{u}_n) \leq M+1$ for all $n$, we may use Step 6 to find $u_n \in \mathscr{J}'_\infty(t_n)$ such that $\leb^1(\{x < 0 : 0 < u_n(x) < 2\}) \leq C_{M,\eta}$ where $C_{M,\eta}$ depends only on $M$ and $\eta$, and $\Psi(u_n) \leq \Psi(\tilde{u}_n) + O(\eta)$, so that
\begin{equation}
\beta(t_n) \leq \Psi(u_n) \leq \beta(t_n) + \frac1n + O(\eta).  \label{eq:psiapprox}
\end{equation}  

Note that $\{u_n\}$ is uniformly bounded in $\Psi$.  In particular, $\sup_{n \in \N} \int_{-m}^0 |u_n''|^2\,dx < \infty$ for each $m \in \N$.  

Consider $m=2$.  By Lemma \ref{psicompact} we may extract a subsequence $\{u_{2,n}\}_n$ of $\{u_n\}_n$ for which $u_{2,n} \wc v_2$ in $W^{2,2}(-2,0)$ for some $v_2 \in W^{2,2}(-2,0)$.  and $u_{2,n} \to v_2$ almost everywhere.

Inductively, for $m \geq 3$ we take $\{u_{m,n}\}_n$ to be a subsequence of $\{u_{m-1,n}\}_n$ such that $u_{m,n} \wc v_m$ in $W^{2,2}(-m,0)$ for some $v_m \in W^{2,2}(-m,0)$ and such that $u_{m,n} \to v_m$ almost everywhere.

We claim that $v_m$ extends $v_{m-1}$ for all $m \geq 3$.  Indeed, $u_{m-1,n} \to v_{m-1}$ almost everywhere, and since $\{u_{m,n}\}_n$ is a subsequence of $\{u_{m-1,n}\}_n$, we have that $u_{m,n} \to v_{m-1}$ almost everywhere over $(-m+1,0)$.  Since $u_{m,n} \to v_m$ almost everywhere, it follows that $v_{m-1}$ and $v_m$ agree over $(-m+1,0)$ as needed.

It follows that there exists a unique $u:(-\infty,0) \to \R$ extending each $v_m$.  Now consider $\{u_{n,n}\}$.  For each $m$, $\{u_{n,n}\}_n$ may be viewed as a subsequence of $\{u_{m,n}\}$ for $n$ large enough, so $u_{n,n} \to v_n = u$ almost everywhere in $(-m,0)$.  Thus $u_{n,n} \to v$ over $(-\infty,0)$ almost everywhere.  Similarly, we see that $u_{n,n} \wc v$ in $W^{2,2}(-m,0)$ for all $m > 0$.  In particular, $u_{n,n} \wc v$ in $W^{2,2}_{\text{loc}}(-m,0)$. 

Since $\{u_{n,n}\}_n$ is a subsequence of $\{u_n\}_n$, we let $n_k$ be such that $\{u_{n_k}\}_k = \{u_{n,n}\}_n$.

To finish we use the fact that $u_{n_k} \to u$ almost everywhere in $(-\infty,0)$ to obtain the inequality
$$\int_{-\infty}^0 (u^2-1)^2\,dx \leq \liminf_{k \to \infty} \int_0^T (u_{n_k}^2-1)^2\,dx$$
by Fatou's Lemma.  Next, we use the property that $u_{n_k}'' \wc u''$ in $L^2_{\text{loc}}(0,T)$ and the fact that $z \mapsto z^2$ is convex, together with Theorem \ref{convexslsc}, to conclude that
$$\int_{-\infty}^0 |u''|^2\,dx \leq \liminf_{k \to \infty} \int_0^T |u_{n_k}''|^2\,dx.$$
Combining these two inequalities gives $\Psi(u) \leq \liminf_{k \to \infty} \Psi(u_{n_k})$

Now, on one hand, we have from \eqref{eq:psiapprox} that $\Psi(u_{n_k}) \leq \beta(t_{n_k}) + \frac{1}{n_k} + O(\eta)$, and taking the liminf we have
$$L + O(\eta) = \liminf_{k \to \infty} \beta(t_{n_k}) + O(\eta) \geq \liminf_{k \to \infty} \Psi(u_{n_k}).$$

On the other hand, we claim that $\Psi(u) \geq \beta(t_0)$.  It suffices to prove that $u \in \overline{\mathscr{J}'_\infty(t_0)}$.  

Since $u_{n_k} \to u$ almost everywhere, and $u_{n_k}(0_-) = t_{n_k}$ with $t_{n_k} \to t_0$, we must have $u(0_-) = t_0$.  Moreover, since $\Psi(u) < \infty$, we have by Step 5 that the limit $\lim_{x \to -\infty} u(x)$ exists, and is either $1$ or $-1$.  However,
$$\leb^1(\{u > 0\}) = \leb^1\left(\liminf_{k \to \infty} \{u_{n_k} > 0\}\right) \leq \liminf_{k \to \infty} \leb^1(\{u_{n_k} > 0\}) \leq C_{M,\eta} < \infty,$$
which eliminates the possibility that the limit is 1.  Thus $\lim_{x \to -\infty} u(x) = -1$, so we may conclude that $u \in \overline{\mathscr{J}'_\infty(t_0)}$.

With $\Psi(u) \geq \beta(t_0)$, we have that
$$L + O(\eta) \geq \liminf_{k \to \infty} \Psi(u_{n_k}) \geq \Psi(u) \geq \beta(t_0).$$
But $\eta > 0$ was arbitrary, so $L \geq \beta(t_0)$, which is what we wanted to show.
\end{proof}

We now prove our main result, Theorem \ref{boundary2}.

\begin{proof}
The cases in which $u \not\in L^p(\Omega,\{-1,1\})$ or $u \in L^p(\Omega,\{-1,1\}) \setminus BPV(\Omega;\{-1,1\})$ are handled as in Steps 1 and 4 of the proof of Theorem \ref{gamma2}, so let us assume that $u \in BPV(\Omega;\{-1,1\})$.

For the liminf inequality, let $u_n \to u$ in $L^p(\Omega)$.  We may assume that $u_n(a_+) = a_{\epsilon_n}$ and $u_n(b_-) = b_{\epsilon_n}$ for all $n$.  

Fix $\eta > 0$, and suppose that $u(b_-) = -1$.  Then we may find an interval $(b-\delta,b)$ over which $u = -1$, and now for all large enough $n$ we follow Step 2 of the proof of Theorem \ref{gamma2} in which we ``modify" $u_n$ in $(b,b-\delta)$ so that it becomes ``affixed" to $-1$ at $b-\delta$.  That is, we find $v_n:(b-\delta,b) \to \R$ such that $v_n \in W^{2,p}(b-\delta,b)$, $v_n(b-\delta_-) = -1$, $v_n'(b-\delta_-) = 0$, $v_n(b_-) = u_n(b_-) = b_{\epsilon_n}$, and
$$\int_{b-\delta}^{b} \epsilon_n^{-1}(v_n^2-1)^2+\epsilon_n^3|v_n''|^2\,dx \leq \int_{b-\delta}^{b} \epsilon_n^{-1}(u_n^2-1)^2+\epsilon_n^3|u_n''|^2\,dx + \eta.$$
Now, as in Step 3 of the proof of Theorem \ref{gamma2}, we may change variables and apply the AM-GM inequality to eventually obtain the bound
$$\int_{b-\delta}^{b} \epsilon_n^{-1}(v_n^2-1)^2+\epsilon_n^3|v_n''|^2\,dx \geq \frac{4}{3^{3/4}}\inf_{u \in \mathscr{J}'(b_{\epsilon_n})} \Phi(u) = 2\beta(b_{\epsilon_n}).$$
If instead $u(b_-) = 1$, then by a symmetrical argument, we instead obtain the term $2\beta(-b_{\epsilon_n})$.  Both of these terms may be written as $2\beta(-\sgn(u(b_-))b_{\epsilon_n})$.  

Using the same argument, we may obtain the term $2\beta(-\sgn(u(a_+))a_{\epsilon_n})$, and we recover the term $\alpha\essVar_\Omega u$ as in Steps 2 and 3 of the proof of Theorem \ref{gamma2}.

Overall, we see that the inequality
$$\liminf_{n \to \infty} G_{\epsilon_n}(u_n) \geq \alpha\essVar_\Omega u + \liminf_{n \to \infty} 2\beta(-\sgn(u(a_+))a_{\epsilon_n}) + 2\beta(-\sgn(u(b_-))b_{\epsilon_n})$$
may be obtained, and so we get
$$\liminf_{n \to \infty} G_{\epsilon_n}(u_n) \geq \alpha\essVar_\Omega u +  2\beta(-\sgn(u(a_+))a_0) + 2\beta(-\sgn(u(b_-))b_0)$$
by continuity of $\beta$.

For the limsup inequality, we use a construction similar to that done in Step 5 of the proof of Theorem \ref{gamma2}.  We begin by strengthening the continuity result on $\beta(t)$.  Define the constant
$$\beta_\epsilon(t) := \inf\left\{\Psi(v) : v \in \mathscr{J}'_\infty(t), L_v \leq \frac{1}{\sqrt{\epsilon}}\right\}$$
and let $t_0 \in \R$ with $t_n \to t_0$.  We claim that $\lim_{n \to \infty} \beta_{\epsilon_n}(t_n) = \beta(t_0)$.  To see this, fix $\eta > 0$ and for each $n \in \N$ take $v_n \in \mathscr{J}'_\infty(t_n)$ for which $\beta(t_n) \leq \Psi(v_n) \leq \beta(t_n) + \frac1n$.  The continuity of $\beta$ ensures that $\{\Psi(v_n)\}_n$ is bounded by a constant $M > 0$.  By Lemma \ref{boundL} we may construct $\tilde{v}_n$ for which $\Psi(\tilde{v}_n) \leq \Psi(v_n) + O(\eta) \leq \beta(t_n) + \frac1n + O(\eta)$ and $L_{\tilde{v}_n} \leq L_{M,\eta}$ where $L_{M,\eta}$ depends only on $M$ and $\eta$.  Particularly $L_{M,\eta}$ has no dependence on $n$, thus $L_{\tilde{n}_n} \leq \frac{1}{\sqrt{\epsilon_n}}$ for all $n$ large enough.  For all such $n$ we may write
$$\beta_{\epsilon_n}(t_n) \leq \Psi(\tilde{v}_n) \leq \beta(t_n) + \frac1n + O(\eta),$$
and so by taking the limsup we may obtain
$$\limsup_{n \to \infty} \beta_{\epsilon_n}(t_n) \leq \beta(t_0) + O(\eta)$$ by continuity of $\beta$.  As $\eta > 0$ was arbitrary, we get $\limsup_{n \to \infty} \beta_{\epsilon_n}(t_n) \leq \beta(t_0)$.  But we clearly also have $\beta_{\epsilon_n}(t_n) \geq \beta(t_n)$, and taking the liminf finishes the proof of the claim.

We now turn to the proof of the limsup inequality.  As we did for the liminf inequality, assume the case $u(b_-) = -1$ and find $\delta > 0$ small enough so that $u = -1$ over the interval $(b-\delta,b)$.  We will define $u_n$ over $(b-\delta,b)$.  Take $v_n \in \mathscr{J}'_\infty(b_{\epsilon_n})$ satisfying $L_{v_n} \leq \frac{1}{\sqrt{\epsilon_n}}$ such that
\begin{equation}
    \beta_{\epsilon_n}(b_{\epsilon_n}) \leq \Psi(v_n) \leq \beta_{\epsilon_n}(b_{\epsilon_n}) + \frac1n. \label{beta2bound}
\end{equation}
Since $\epsilon_nL_{v_n} \leq \sqrt{\epsilon_n} \to 0$, we have $\epsilon_nL_{v_n} < \delta$ for all $n$ large enough.  For all such $n$, we define
$$u_n(x) := \begin{cases}-1, & b-\delta < x \leq b-\epsilon_nL_{v_n} \\ v_n\left(\frac{x-b}{\epsilon_n}\right), & b-\epsilon_nL_{v_n} < x < b\end{cases}.$$
This satisfies $u_n(b-\delta) = -1$, $u_n'(b-\delta) = 0$, and the boundary condition $u_n(b_-) = b_{\epsilon_n}$.  Moreover
\begin{align*}
  \int_{b-\delta}^b \epsilon_n^{-1}(u_n^2-1)^2 &+ \epsilon_n^3|u_n''|^2\,dx\\
  &= \int_{b-\epsilon_nL_{v_n}}^b  \epsilon_n^{-1}\left(v_n\left(\frac{x-b}{\epsilon_n}\right)^2-1\right)^2 + \frac{\epsilon_n^3}{\epsilon_n^4}v_n''\left(\frac{x-b}{\epsilon_n}\right)^2\,dx\\
  &= \int_{-L_{v_n}}^0 (v_n^2-1)^2 + |v_n''|^2\,dx\\ 
  &= \Psi(v_n) \leq \beta_{\epsilon_n}(b_{\epsilon_n})+\frac1n,
\end{align*}
thus
$$\limsup_{n \to \infty} \int_{b-\delta}^b \epsilon_n^{-1}(u_n^2-1)^2 + \epsilon_n^3|u_n''|^2\,dx \leq \limsup_{n \to \infty} \beta_{\epsilon_n}(b_{\epsilon_n})+\frac1n = \beta(b_0)$$
by the claim.  It remains to show that $u_n \to -1$ in $L^p(b-\delta,b)$. 

Since $\epsilon_nL_{v_n} \to 0$, we have that $u_n \to -1$ almost everywhere in $(b-\delta,b)$.  By Vitali, it suffices to prove that $\{|u_n|^p \cdot 1_{(b-\delta,b)}\}_n$ is uniformly integrable.  Indeed, we have
$$\int_{b-\delta}^b (u_n^2-1)^2\,dx = \epsilon_n\int_{-L_{v_n}}^0 (v_n^2-1)^2\,dx = \epsilon_n\Psi(v_n),$$
and since $\limsup_{n \to \infty} \Psi(v_n) \leq \lim_{n \to \infty} \beta_{\epsilon_n}(b_{\epsilon_n}) + \frac1n < \infty$, we have that $\lim_{n \to \infty} \int_{b-\delta}^b (u_n^2-1)^2\,dx = 0$, so uniform integrability follows from Lemma \ref{unif}.

Our construction for $u_n$ over $(b-\delta,b)$ provides us with the term $\beta(b_0)$ for the Gamma limit in the case that $u(b_-) = -1$.  If instead $u(b_-) = 1$, we may negate our construction to instead acquire the term $\beta(-b_0)$.   The term in both cases is equal to $\beta(b_0\sgn u(b_-))$.  A symmetrical construction near the endpoint $a$ yields the term $\beta(a_0)\sgn u(a_+))$.  Lastly we may define $u_n$ away from the endpoints as in Step 5 of the proof of Theorem \ref{gamma2} in order to recover the term $\alpha \essVar_\Omega u$, completing the proof. 
\end{proof}

\section{Acknowledgments}

This research was partially supported by the National Science Foundation under grants No. DMS 2108784 and DMS 1714098.  I am also indebted to my advisor Giovanni Leoni, whose guidance and feedback made this paper possible.

\bibliographystyle{abbrv}
\bibliography{bibliography.bib}

\end{document}